\documentclass[11pt,a4paper]{amsart}

\usepackage[utf8]{inputenc}
\usepackage{amsmath,amssymb,amsthm}
\usepackage{mathrsfs}
\usepackage{enumitem}
\usepackage[colorlinks=true,linkcolor=blue,citecolor=blue,urlcolor=blue,bookmarks=false]{hyperref}
\usepackage[margin=2.5cm]{geometry}

\newtheorem{theorem}{Theorem}[section]
\newtheorem{corollary}[theorem]{Corollary}
\newtheorem{lemma}[theorem]{Lemma}
\newtheorem{proposition}[theorem]{Proposition}

\theoremstyle{definition}
\newtheorem{definition}[theorem]{Definition}
\newtheorem{example}[theorem]{Example}
\newtheorem{remark}[theorem]{Remark}

\title{The quaternionic moment problem}

\author{R. Ben Taher}
\address{Department of Mathematics, Faculty of Sciences, Moulay Ismail University, Meknes, Morocco}
\email{r.bentaher@umi.ac.ma}

\author{K. Schm\"udgen}
\address{Department of Mathematics, University of Leipzig, Leipzig, Germany}
\email{Konrad.Schmuedgen@math.uni-leipzig.de}

\author{E.H. Zerouali}
\address{Department of Mathematics, Faculty of Sciences, Mohammed V University, Rabat, Morocco}
\email{elhassan.zerouali@fsr.um5.ac.ma}


\begin{document}

\begin{abstract}
In this paper we develop an approach to the full quaternionic moment problem.
We define a hierarchy
$ \mathbb{H}^{k}[q^{*}, q]\subset \mathbb{H}^{k+1}[q^{*}, q]$,
$k\in \mathbb{N}_{0}\cup \{\infty \}$, of two-sided $\mathbb{H}$-linear
spaces of quaternionic polynomials which are invariant under conjugation
of quaternions and determine the hermitian parts of these spaces explicitly.
Using a generalization of Choquet's theorem on adapted spaces to quaternions
we provide necessary and sufficient solvability criteria for the quaternionic
moment problem of each space $ \mathbb{H}^{k}[q^{*}, q]$. The hermitian
part of $ \mathbb{H}^{\infty }[q^{*}, q]$ is the real polynomial algebra
$\mathbb{R}[x_{0}, x_{1}, x_{2}, x_{3}]$. This enables us to apply real
algebraic geometry (Positivstellens\"{a}tze) to the quaternionic moment problem
on $ \mathbb{H}^{\infty }[q^{*}, q]$.
\end{abstract}

\maketitle

\section{Introduction}
\label{sec1}

The moment problem, one of the central topics in analysis and probability
theory, seeks to determine when a given sequence of real numbers can be
realized as the moments of some positive measure. This fundamental question
was first formulated by the Dutch mathematician Th. J. Stieltjes in his
pioneering memoir \cite{Stieltjes1894}. Now there exists a well--developed
theory of moment problems which has deep connections with many other mathematical
disciplines such as spectral theory, complex analysis, orthogonal polynomials,
real algebraic geometry, optimization, and others. A comprehensive treatment
of the various forms of the classical moment problem is given in the monograph
\cite{Schmudgen2017}, see also \cite{Akhiezer1965}.

In the present paper we study the moment problem over quaternions. More
precisely, our aim is to develop fundamental notions and basic solvability
criteria for the quaternionic moment problem. To the best of our knowledge,
a systematic study of quaternionic moment problems has not yet been explored
in the existing literature.

The algebra $\mathbb{H}$ of quaternions, discovered by Hamilton in 1843
\cite{Hamilton1843}, is a four-dimensional associative non-commutative
division algebra over the real numbers. Quaternions have been investigated
and applied in several fields of modern mathematics and physics. As examples,
we mention the study of operators in Hilbert spaces over quaternions
\cite{CGK}, \cite{vasi22} and the analysis of quaternionic functions
\cite{GSS}. This motivates the development of a quaternionic moment theory.

Moment problems in modern functional analytic language deal with the representation
of linear functionals on polynomials as integrals with respect to positive
measures. In the case of quaternions a number of new difficulties and challenges
occur. An obvious difficulty is the non-commutativity of quaternions. A
more important challenge is the fact that important classes of quaternionic
polynomials and their hermitian parts are not invariant under multiplication,
so they do not form algebras. In particular, since the hermitian parts
are not algebras, in contrast to the classical case methods of real algebraic
geometry do not apply directly. Our approach might open new perspectives
for future research in non-commutative moment theory and real algebraic
geometry.

Our starting point to the quaternionic moment problem is a chain of spaces
of quaternionic polynomials:
\begin{align}
\mathbb{H}^{0}[q^{*}, q]\subset \cdots \subset \mathbb{H}^{k}[q^{*}, q]
\subset \mathbb{H}^{k+1}[q^{*}, q]\subset \mathbb{H}^{\infty}[q^{*}, q].
\label{eq1}
\end{align}
These are two-sided $\mathbb{H}$-linear vector spaces (or equivalently,
$\mathbb{H}$-bimodules) which are invariant under the conjugation
\begin{equation*}
q = x_{0} + x_{1} \mathrm{i}+ x_{2} \mathrm{j}+ x_{3} \mathrm{k}\to q^{*}:=x_{0}-x_{1}
\mathrm{i}-x_{2}\mathrm{j}-x_{3}\mathrm{k},
\end{equation*}
where $x_{0}, x_{1}, x_{2}, x_{3} \in \mathbb{R}$ and
$\mathrm{i}, \mathrm{j}$, $\mathrm{k}$ denote the fundamental units of
$\mathbb{H}$.

The elements of the space $\mathbb{H}^{0}[q^{*}, q]$ of basic polynomials
are finite sums of terms $a(q^{*})^{m}q^{n}b$, where
$a,b \in \mathbb{H}$ and $m,n\in \mathbb{N}_{0}$, and the elements of
$ \mathbb{H}^{k+1}[q^{*}, q]$ are sums of products $f\cdot g$, with
$f\in \mathbb{H}^{k}[q^{*}, q]$ and $g\in \mathbb{H}^{0}[q^{*}, q]$.

Then we prove that the hermitian parts of these $*$-invariant two-sided
$\mathbb{H}$-linear vector spaces are
\begin{align}
\label{hhherkpart}
\mathbb{H}^{k}[q^{*},q]_{\mathrm{her}}=\sum _{\alpha _{1},\alpha _{2},
\alpha _{3}\in \mathbb{N}_{0}, \, \alpha _{1}+\alpha _{2}+\alpha _{3}
\leq k+1} x_{1}^{\alpha _{1}}x_{2}^{\alpha _{2}}x_{3}^{\alpha _{3}}\,
\, \mathbb{R}[x_{0},x_{1}^{2}+x_{2}^{2}+x_{3}^{3}]
\end{align}
for $k\in \mathbb{N}_{0}$ and
\begin{align}
\label{hinftyxherpart}
\mathbb{H}^{\infty}[q^{*},q]_{\mathrm{her}}=\mathbb{R}[x_{0},x_{1},x_{2},x_{3}].
\end{align}

Our approach to the quaternionic moment problem is based on a generalization
of Choquet's theorem on adapted spaces to the case of quaternions. This
result will be then applied to the polynomial spaces
$ \mathbb{H}^{k}[q^{*}, q]$.

For the solvability theory of the classical multi-dimensional moment problem
there exist powerful methods based on Archimedean Positivstellens\"atze
of real algebraic geometry (see \cite[Chapter 13]{Schmudgen2017}). These
results require a finitely generated real algebra. By (\ref{hinftyxherpart})
the hermitian part of the full polynomial space
$ \mathbb{H}^{\infty}[q^{*}, q]$ is the polynomial algebra
$\mathbb{R}[x_{0},x_{1},x_{2},x_{3}]$. Then Archimedean Positivstellens\"atze
can be applied in this case and a theory similar to the real multi-dimensional
moment problem can be developed. However, if $k\in \mathbb{N}_{0}$, then
the hermitian part of $ \mathbb{H}^{k}[q^{*}, q]$, as described by (\ref{hhherkpart}),
is neither a real algebra nor a quadratic module or a semiring for some
reasonable real algebra. It is an interesting challenge and open problem
in real algebraic geometry whether there exist counter-parts of the Archimedean
Positivstellens\"atze for the real vector spaces (\ref{hhherkpart}) of
polynomials. This is of particular interest in the case of the space
$ \mathbb{H}^{0}[q^{*}, q]$ of basic polynomials.

Let us briefly explain the structure of this paper and its main results.

In Section~\ref{secadapted}, we generalize Choquet's concept of adapted spaces
to quaternions and prove the counter-part of Choquet's theorem for quaternionic
moment functionals (Theorem~\ref{choquet}). As a byproduct we give a quaternionic
version of the Richter-Tchakaloff theorem of the existence of atomic representing
measures for moment functionals on finite dimensional spaces (Corollary~\ref{richter}).

In Section~\ref{polynomialspaces}, we define the quaternionic polynomial
spaces $ \mathbb{H}^{k}[q^{*}, q]$, where
$k\in \mathbb{N}_{0}\cup \{\infty \}$, and study them in detail. The main
results are the explicit descriptions of their hermitian parts given by
the formulas (\ref{hhherkpart}) and (\ref{hinftyxherpart}) above (Theorems~\ref{herhk} and \ref{herinfty}). Further, we show that these polynomial
spaces are adapted spaces in the sense of Choquet (Proposition~\ref{adaptedH_L}).

In Section~\ref{mpbasic}, we treat the moment problem for the space
$ \mathbb{H}^{0}[q^{*}, q]$ of basic polynomials (Theorem~\ref{qmp1}) and
derive a solvability criterion for a class of specific semi-algebraic sets
(Theorem~\ref{mpK_h}).

In Section~\ref{mphk} we investigate the moment problem for the general
polynomial spaces $ \mathbb{H}^{k}[q^{*}, q]$,
$k\in \mathbb{N}_{0}\cup \{\infty \}$, and derive a quaternionic version
of Haviland's theorem (Theorem~\ref{thm:quaternion_haviland}).

Section~\ref{mpinfty} deals with the moment problem for the full polynomial
space $ \mathbb{H}^{\infty}[q^{*}, q]$. As noted above, in this case the
hermitian part is the polynomial algebra
$\mathbb{R}[x_{0},x_{1},x_{2},x_{3}]$ and Archimedean Positivstellens\"atze
can be applied. We carry out this for the Positivstellensatz obtained in
\cite{Schmudgen1991} and formulate its counter-part (Theorem~\ref{mpppreordering}).

In Section~\ref{prelim} we collect basic facts on the quaternions and fix
some notation which will be used throughout this paper.

\section{Preliminaries on quaternions}
\label{prelim}

The algebra of quaternions is denoted by $\mathbb{H}$. Each quaternion
$q \in \mathbb{H}$ can be uniquely written as
\begin{equation*}
q = x_{0} + x_{1} \mathrm{i}+ x_{2} \mathrm{j}+ x_{3} \mathrm{k},
\end{equation*}
where $x_{0}, x_{1}, x_{2}, x_{3} \in \mathbb{R}$ and
$\mathrm{i}, \mathrm{j}$, $\mathrm{k}$ are the fundamental units of
$\mathbb{H}$ satisfying
\begin{align}
\label{qxx}
\mathrm{i}^{2}= \mathrm{j}^{2}=\mathrm{k}^{2}=-1,~~ \mathrm{i}
\mathrm{j}=-\mathrm{j}\mathrm{i}=\mathrm{k}, \mathrm{j}\mathrm{k}=-
\mathrm{k}\mathrm{j}=\mathrm{i}, \mathrm{k}\mathrm{i}=-\mathrm{i}
\mathrm{k}=\mathrm{j}.
\end{align}
We will identify the quaternion $q\in \mathbb{H}$ and the element
$(x_{0},x_{1},x_{2},x_{2})\in \mathbb{R}^{4}$. Then
$\mathbb{H}\cong \mathbb{R}^{4}$ becomes a locally compact space.

For
$q = x_{0} + x_{1} \mathrm{i}+ x_{2} \mathrm{j}+ x_{3} \mathrm{k}$, we
define its conjugate by
\begin{align}
\label{qstar}
q^{*}:=x_{0}-x_{1}\mathrm{i}-x_{2}\mathrm{j}-x_{3}\mathrm{k}
\end{align}
and its norm by
\begin{equation*}
\|q\|=\sqrt{q^{*}q}=\sqrt{x_{0}^{2}+x_{1}^{2}+x_{2}^{2}+x_{3}^{2}}.
\end{equation*}
For each $q\in \mathbb{H}$ we have the conjugation formula
\begin{align}
\label{conjugationformula}
q^{*}=-\frac{1}{2}(q+\mathrm{i}q \mathrm{i}+\mathrm{j}q\mathrm{j}+
\mathrm{k}q\mathrm{k}).
\end{align}
From (\ref{qxx}) and (\ref{qstar}) we obtain
\begin{equation}
\label{coordinates}
x_{0} = \frac{q + q^{*}}{2},~~ x_{1} =
\frac{q^{*}\mathrm{i}- \mathrm{i}q}{2},~~ x_{2} =
\frac{q^{*}\mathrm{j}-\mathrm{j}q}{2},~~ x_{3} =
\frac{q^{*}\mathrm{k}- \mathrm{k}q}{2}.
\end{equation}
For more details on quaternion algebra and linear algebra, see
\cite{rodman2014}.

Let us consider the set $C(\mathbb{H}; \mathbb{H})$ of $\mathbb{H}$-valued
continuous functions on $\mathbb{H}$. Clearly,
$C(\mathbb{H};\mathbb{H})$ is a unital algebra over the quaternions with
pointwise multiplication. The map $f\to f^{*}$ of
$C(\mathbb{H}; \mathbb{H})$ into itself, defined by
\begin{equation}
\label{defstar}
f^{*}(q):= f(q)^{*}\quad \text{for}~~ q\in \mathbb{H},
\end{equation}
satisfies
\begin{equation}
\label{invprop}
(af+bg)^{*}=f^{*}a^{*}+g^{*}b^{*},\quad (fg)^{*}=g^{*}f^{*}, \quad (f^{*})^{*}=f
\end{equation}
for $f,g\in C(\mathbb{H}; \mathbb{H})$ and $a,b\in \mathbb{H}$. That is,
the map $f\to f^{*}$ is an involution satisfying the compatibility conditions
$(fg)^{*}=g^{*}f^{*}$ with the multiplication. We call such a structure
on $C(\mathbb{H}; \mathbb{H})$ a $*$-algebra over the quaternions or briefly
a quaternionic $*$-algebra. In particular, if we consider
$C(\mathbb{H}; \mathbb{H})$ only as an algebra over the reals, then
$C(\mathbb{H}; \mathbb{H})$ is a real $*$-algebra (see
\cite[Definition 2.1]{Schmudgen2020}). The definition (\ref{defstar}) and
the properties (\ref{invprop}) will be used in the sequel without mention.

Suppose that $E$ is a two-sided $\mathbb{H}$-linear subspace of
$C(\mathbb{H};\mathbb{H})$, that is,
\begin{equation*}
afb+ cgd\in E\quad \text{for all} \quad f,g\in E , ~a,b,d,c\in
\mathbb{H}.
\end{equation*}
We say that $E$ is $*$-invariant if $f\in E$ implies that
$f^{*}\in E$.

Suppose now that $E$ is a $*$-invariant real linear subspace of
$C(\mathbb{H};\mathbb{H})$. We define the \emph{hermitian part} of
$E$ by
\begin{equation*}
E_{\mathrm{her}}:= \{ f\in E: f=f^{*}\}.
\end{equation*}
Further, for $f\in E$ we have
\begin{equation*}
f=\mathrm{Re}f+ \mathrm{Im}f,
\end{equation*}
where
\begin{align*}
\mathrm{Re}f:=\frac{f+f^{*}}{2}\quad \textrm{and} \quad \mathrm{Im}f:=
\frac{f-f^{*}}{2}.
\end{align*}
Clearly, $f\in E_{\mathrm{her}}$ if and only if $f=\mathrm{Re}f$ for each
$f\in E$. Hence
\begin{align}
\label{EherRe}
E_{\mathrm{her}}=\{\, \mathrm{Re}f: \, f\in E\}.
\end{align}
From the definition (\ref{defstar}) it follows at once that
$ E_{\mathrm{her}}$ is just the set of real-valued functions on
$\mathbb{H}$:
\begin{equation}
\label{chareher}
E_{\mathrm{her}}= \{ f\in E: f(q)\in \mathbb{R}\quad \text{for}~~ q
\in \mathbb{H}\, \}.
\end{equation}
Clearly, (\ref{chareher}) implies that $E_{\mathrm{her}}$ is an
$\mathbb{R}$-linear subspace contained in $C(\mathbb{H};\mathbb{R})$.

Note that the preceding definitions and facts apply in particular to each
$*$-invariant two-sided $\mathbb{H}$-linear subspace $E$ of
$C(\mathbb{H};\mathbb{H})$.
\begin{lemma}%
\label{commutative}
If $E$ is a $*$-subalgebra of $C(\mathbb{H};\mathbb{H})$, then
$E_{\mathrm{her}}$ is a commutative real subalgebra of
$C(\mathbb{H}; \mathbb{R})$.
\end{lemma}
\begin{proof}
Let $f,g \in E_{\mathrm{her}}$. Since $f$ and $g$ are real-valued by (\ref{chareher}),
we have
\begin{align}
\label{fgqq}
(fg)(q)=f(q)g(q)=g(q)f(q)=(gf)(q)\quad \text{for all}~~ q\in
\mathbb{H}.
\end{align}
Since $E$ is an algebra, $fg\in E$. From (\ref{fgqq}) we conclude that
$fg$ is real-valued, hence $fg\in E_{\mathrm{her}}$, and that the mappings
$fg$ and $gf$ of $C(\mathbb{H}, \mathbb{H})$ coincide. Thus $fg=gf$. As
noted above, $E_{\mathrm{her}}$ is a real vector space. Therefore,
$E_{\mathrm{her}}$ is a commutative real algebra.
\end{proof}

For a subset $F$ of $C(\mathbb{H};\mathbb{H})$, we denote by
${\mathrm{span}}_{\mathbb{H}}\, F$ the smallest two-sided $\mathbb{H}$-linear
subspace of $C(\mathbb{H};\mathbb{H})$ which contains $F$, that is,
${\mathrm{span}}_{\mathbb{H}}\, F$ is the set of finite sums of elements
$afb$, where $a,b\in \mathbb{H}$ and $f\in F$.

\section{Adapted spaces and quaternionic moment functionals}
\label{secadapted}

Suppose that $X$ be a locally compact topological Hausdorff space and
$E_{\mathbb{R}}$ be an $\mathbb{R}$-linear subspace of
$C(X;\mathbb{R})$. Let
\begin{align}
\label{defe+}
E_{+}:=\{ f\in E_{\mathbb{R}}:~ f(x)\geq 0 \quad \text{for} \quad x
\in X\}
\end{align}
First we recall Choquet's definition of an adapted space. See
\cite[Definition 1.5]{Schmudgen2017} and Choquet's theorem in
\cite{Choquet1969}.
\begin{definition}[Adapted Space]%
\label{adapted}
 The space $E_{\mathbb{R}}$ of $C(X;\mathbb{R})$ is called
an \emph{adapted space} if the following conditions are satisfied:
\begin{enumerate}
\item[\emph{(i)}] $ E_{\mathbb{R}}=E_{+}-E_{+}$.
\item[\emph{(ii)}] For each $x\in X$ there exists an $f\in E_{+}$ such that
$f(x)>0$.
\item[\emph{(iii)}] For each $f\in E_{+}$ there exists a $g\in E_{+}$ that
dominates $f$ in the sense that for any $\varepsilon >0$ there is a compact
subset $K_{\varepsilon}$ of $X$ such that
$f(x)\leq \varepsilon g(x)$ for all
$x\in X\backslash K_{\varepsilon}$.
\end{enumerate}
\end{definition}

Note that if $X$ is locally compact, but not compact, then by the Alexandroff
compactification there exists a point ``$\infty $'' such that
$X\cup \{\infty \}$ is compact. That $g$ dominates $f$ as in condition
(iii) means that $\frac{f(x)}{g(x)}\to 0$ as $x\to \infty $.
\begin{theorem}[Choquet's Theorem]%
\label{choquet}%
 Suppose that $E_{\mathbb{R}}$ is an adapted subspace
of $C(X;\mathbb{R})$. For each linear functional
$L:E_{\mathbb{R}}\to \mathbb{R}$ the following statements are equivalent:
\begin{enumerate}
\item[\emph{(i)}] The functional $L$ is $E_{+}$-positive, that is,
$L(f)\geq 0$ for all $f\in E_{+}$.
\item[\emph{(ii)}] For each $f\in E_{+}$ there exists $h\in E_{+}$ such that
$L(f+\varepsilon h)\geq 0$  for all $\varepsilon >0$.
\item[\emph{(iii)}] $L$ is a moment functional, that is, there exists a (positive)
Radon measure $\mu $ on $X$ such that
\begin{equation*}
L(f) = \int _{X}\, f(x) d\mu (x)\quad \text{for all}\quad f \in E_{
\mathbb{R}}.
\end{equation*}
\end{enumerate}
\end{theorem}
\begin{proof}
\cite{Choquet1969} or \cite[Theorem 1.8]{Schmudgen2017}.
\end{proof}

Next we develop a version of this result for the quaternionic moment problem.
For this we need some preliminaries and notation.

\begin{definition}
\label{defn3.3}
An \emph{$\mathbb{H}$-linear functional} on a two-sided $\mathbb{H}$-vector
space $E$ is a mapping $L:E\to \mathbb{H}$ such that
\begin{equation*}
L(afb+cgd)=aL(f)b +cL(f)d\quad \text{for all} \quad f,g\in E,~ a,b,c,d
\in \mathbb{H}.
\end{equation*}
\end{definition}
Suppose that $K$ is a closed subset of $\mathbb{H}$. Then $K$ is a locally
compact topological space in the induced topology from $\mathbb{H}$. From
(\ref{chareher}) it follows that the following definition makes sense:%
\begin{align}
E(K)_{+}:=\{ f\in E_{\mathrm{her}}:~ f(x)\geq 0 \quad \text{for}
\quad x\in K\}.
\end{align}
\begin{definition}%
\label{quatmf}
An $\mathbb{H}$-linear functional $L:E\to \mathbb{H}$ on a $*$-invariant
two-sided $\mathbb{H}$-vector space $E$ is called a
\emph{$K$-moment functional} on $E$ if there exists Radon measure
$\mu $ on the locally compact space $K$ such that
\begin{equation}
L(f) = \int _{K}\, f(x) d\mu (x)\quad \text{for all}\quad f \in E.
\end{equation}
\end{definition}
The following technical result is crucial for our study of the quaternionic
moment problem.
\begin{proposition}%
\label{extensionherHH}
Let $K$ be a closed subset of $\mathbb{H}$ and let $E$ be a two-sided
$\mathbb{H}$-linear subspace of  $C(K;\mathbb{H})$ which is $*$-invariant.
Let $L:E\to \mathbb{H}$ be an $\mathbb{H}$-linear functional on $E$. Suppose
that there exists a positive Radon measure $\mu $ on $K$ such that
\begin{equation}
\label{inteh1}
L(f) = \int _{K}\, f(x) d\mu (x)
\end{equation}
for all $f \in E_{\mathrm{her}}$. Then equation (\ref{inteh1}) holds for
all $f\in E$.
\end{proposition}
\begin{proof}
Let $f\in E$. We set
\begin{align*}
f_{0}:=\frac{f+f^{*}}{2}, ~~ f_{1}:=
\frac{- \mathrm{i}f+f^{*} \mathrm{i}}{2}, ~~ f_{2}:=
\frac{- \mathrm{j}f+f^{*}\mathrm{j}}{2},~~ f_{3}:=
\frac{- \mathrm{k}f+f^{*}\mathrm{k}}{2}\, .
\end{align*}
Because $E$ is a $*$-invariant two-sided $\mathbb{H}$-linear space, the
functions $f_{0}, f_{1},f_{2},f_{3}$ belong to $E$, hence to
$E_{\mathrm{her}}$, and we have the decomposition
\begin{equation*}
f = f_{0} + \mathrm{i}\, f_{1} + \mathrm{j}\, f_{2} + \mathrm{k}\, f_{3}.
\end{equation*}
By assumption, the four functions $f_{0}, f_{1},f_{2},f_{3}$ are
$\mu $-integrable. Hence $f$ is $\mu $-integrable. Then, using the
$\mathbb{H}$-linearity of the functional $L$, equation (\ref{inteh1}) for
$f_{0},f_{1},f_{2}, f_{3}$ and finally the $\mathbb{H}$-linearity of the
integral we derive
\begin{equation*}
\begin{aligned}
L(f) &= L(f_{0} + \mathrm{i}f_{1} + \mathrm{j}f_{2} + \mathrm{k}f_{3})
\\
&= L(f_{0}) + \mathrm{i}L(f_{1}) + \mathrm{j}L(f_{2}) + \mathrm{k}L(f_{3})
\\
&= \int _{K} f_{0} \, d\mu +\mathrm{i}\int _{K} f_{1} \, d\mu +
\mathrm{j}\int _{K} f_{2} \, d\mu + \mathrm{k}\int _{K} f_{3} \, d
\mu
\\
& =\int _{K} (f_{0} + \mathrm{i}f_{1}+ \mathrm{j}f_{2}+ \mathrm{k}f_{3})
\, d\mu = \int _{K} f \, d\mu .
\end{aligned}
\end{equation*}
This proves that (\ref{inteh1}) holds for $f\in E$.
\end{proof}

As an immediate consequence we obtain the Richter-Tchakaloff theorem for
the quaternionic truncated moment problem.
\begin{corollary}%
\label{richter}
Let $K$, $E$, $L$, and $\mu $ be as in Proposition~\ref{extensionherHH}. Suppose that the vector space $E$ is finite dimensional.
Then there exists a finitely atomic measure
$\nu =\sum _{j=1}^{n} c_{j}\delta _{x_{j}}$, where $c_{j}\geq 0$ and
$x_{j}\in K$ for $j=1,\dots ,n$ and $n\leq \dim E_{\mathrm{her}}$, such
that $\nu $ is a representing measure for $L$, that is,
\begin{align}
\label{rtt}
L(f)=\int _{K} f(x)\, d\nu (x)\equiv \sum _{j=1}^{n} c_{j} f(x_{j})
\quad \textrm{for}~~~ f\in E.
\end{align}
\end{corollary}
\begin{proof}
We apply the classical Richter-Tchakaloff theorem (see e.g.,
\cite[Theorem 1.24]{Schmudgen2017}) to the real-valued functional
$L$, defined by (\ref{inteh1}), on the finite-dimensional vector space
$E_{\mathrm{her}}$ of real-valued functions on $K$. By this theorem, there
exists a finitely atomic measure
$\nu =\sum _{j=1}^{n} c_{j}\delta _{x_{j}}$, with
$n\leq \dim E_{\mathrm{her}}$, such that $L(f)=\int _{K} f d\nu $ for
$f\in E_{\mathrm{her}}$. Then Proposition~\ref{extensionherHH} implies
that the latter equality holds for all $f\in E$ which gives (\ref{rtt}).
\end{proof}
\begin{theorem}[Quaternionic Choquet Theorem]%
\label{quatchoquet}%
 Let $K$ be a closed subset of
$\mathbb{H}$. Suppose that $E$ is a two-sided $*$-invariant
$\mathbb{H}$-linear subspace of $C(K; \mathbb{H})$ such that
$E_{\mathrm{her}}$ is an adapted subspace according to Definition~\ref{adapted}. Let $L:E \to \mathbb{H}$ be an $\mathbb{H}$-linear functional
on $E$. Then the following are equivalent:
\begin{enumerate}
\item[\emph{(i)}] The functional $L$ is $E_{+}$-positive, that is,
$L(f)\geq 0$ for all $f\in E_{+}$.
\item[\emph{(ii)}] For each $f\in E_{+}$ there exists an $h\in E_{+}$ such
that $L(f+\varepsilon h)\geq 0$  for all $\varepsilon >0$.
\item[\emph{(iii)}] $L$ is a $K$-moment functional on $E$, that is, there
exists a positive Radon measure $\mu $ on $K$ such that
\begin{equation}
\label{inteh}
L(f) = \int _{K}\, f(x) d\mu (x)\quad \text{for all}\quad f \in E.
\end{equation}
\end{enumerate}
\end{theorem}
\begin{proof}
The equivalence (i)$\iff $(ii) and the implication (iii)$\Rightarrow $(i)
are clear. Let us prove the main implication (i)$\Rightarrow $(iii).

Let $L_{\mathrm{her}}$ denote the restriction of $L$ to
$E_{\mathrm{her}}$. (i) implies, in particular, that
$L_{\mathrm{her}}(f)\in \mathbb{R}$ for $f\in E_{+}$. From
$E_{\mathrm{her}}=E_{+}-E_{+}$ (by Definition~\ref{adapted}) it follows
that $L_{\mathrm{her}}$ is real-valued. Since $E_{\mathrm{her}}$ is an
adapted $\mathbb{R}$-linear subspace of $C(K; \mathbb{R})$, Theorem~\ref{choquet} applies to the functional
$L_{\mathrm{her}}:E_{\mathrm{her}}\to \mathbb{R}$ and yields that there
exists a Radon measure $\mu $ on the subset $K$ of $\mathbb{H}$ such that
\begin{equation*}
L(f)= L_{\mathrm{her}}(f) = \int _{K}\, f(x) d\mu (x)\quad
\text{for all}\quad f \in E_{\mathrm{her}}.
\end{equation*}
Now from Proposition~\ref{extensionherHH} we conclude that (\ref{inteh})
holds for all $f\in E$.%
\end{proof}

\section{Quaternionic polynomial spaces}
\label{polynomialspaces}

All quaternionic polynomials defined in the sequel will be considered as
elements of $C(\mathbb{H}; \mathbb{H})$ and the algebraic structures on
quaternionic polynomials are inherited from the quaternionic $*$-algebra
$C(\mathbb{H}; \mathbb{H})$.

\subsection{The basic quaternionic polynomial space}
\label{sec4.1}

We begin with the smallest of our hierarchy of spaces of quaternionic polynomials.
\begin{definition}%
\label{defih0}
\begin{equation*}
\begin{array}{l@{\quad}l@{\quad}l}
\mathbb{H}^{0}[q^{*}, q] &=& \Big\{f= \sum _{j,k,l} a_{jkl}\, (q^{*})^{k}
q^{l}\, b_{jkl} : j, k,l \in \mathbb{N}_{0}~ \textrm{ and }~ a_{jkl}, b_{jkl}
\in \mathbb{H}\Big\}.
\end{array}
\end{equation*}
\end{definition}

Since $q^{*}=\|q\|^{2}\, q^{-1}$ and $\|q\|$ commutes with $q$, it follows
that
\begin{align}
\label{qmqngenerators}
\big( (q^{*})^{k} q^{l}\big)\big( (q^{*})^{m} q^{n} \big)= (q^{*})^{k+m}q^{n+l}
, \quad k,l,m,n\in \mathbb{N}.
\end{align}
Therefore, the generators $(q^{*})^{k} q^{l}$ of the $\mathbb{H}$-linear
space $\mathbb{H}^{0}[q^{*}, q]$ pairwise commute, but the scalar multiplication
of $(q^{*})^{k} q^{l}$ with $a_{jkl}, b_{jkl}\in \mathbb{H}$ is not commutative
in general.

Obviously, $\mathbb{H}^{0}[q^{*}, q]$ is a two-sided $\mathbb{H}$-linear
subspace of $C(\mathbb{H};\mathbb{H})$. Further, since
$f^{*}= \sum _{j,k,l} (b_{jkl})^{*}\, (q^{*})^{l} q^{k}\, (a_{jkl})^{*}$,
the space $\mathbb{H}^{0}[q^{*}, q]$ is also invariant under the involution
$f\to f^{*}$. That is, $\mathbb{H}^{0}[q^{*}, q] $ is a $*$-invariant two-sided
$\mathbb{H}$-linear subspace of $C(\mathbb{H};\mathbb{H})$.

Now we look for a simpler representation of elements of
$\mathbb{H}^{0}[q^{*}, q]$. Set
\begin{align}
\label{defSigma}
\Sigma :=\{ 1, \  \mathrm{i}, \ \mathrm{j}, \  \mathrm{k}\}.
\end{align}
Let $ \mathcal{N}$ denote the set of all triples $(k,n,\sigma )$, where
$k,n\in \mathbb{N}_{0}, \sigma \in \Sigma $ for $n>k$, and $(n,n,1)$, where
$n\in \mathbb{N}_{0}$.

The following result is used in the proof of the next theorem.
\begin{lemma}%
\label{fqzerolemma}
Consider a finite sum
\begin{align*}
f(q)=\sum _{p=0}^{n}\sum _{\sigma \in \Sigma} a_{p\sigma} q^{p}
\sigma ,\quad \textrm{where} \quad a_{p\sigma}\in \mathbb{H}, n\in
\mathbb{N}_{0}.
\end{align*}
If $f(q)=0$ for all $q\in \mathbb{S}:=\{ q\in \mathbb{H}: \|q\|=1\}$, then
$a_{p\sigma}=0$ for all $p,\sigma $.
\end{lemma}
\begin{proof}
Let $\tau \in \{\mathrm{i}, \mathrm{j}, \mathrm{k}\}$ and
$\theta \in [0,2\pi ]$. Set $q=\cos \theta +\tau \sin \theta $. Then
$q\in \mathbb{S}$ and hence
\begin{align*}
0 =&f(q)+f(q^{-1})= 2b_{0}+\sum _{k=1}^{n}\sum _{\sigma \in \Sigma } b_{k
\sigma} (q^{k}+q^{-k})\sigma
\\
=&2b_{0}+ 2 \sum _{k=1}^{n}\sum _{\sigma \in \Sigma } b_{k\sigma}(
\cos (k \theta ))\sigma =2b_{0}+ 2\sum _{k=1}^{n}\left [\sum _{
\sigma \in \Sigma } b_{k\sigma}\sigma \right ] \cos (k\theta ).
\end{align*}
Using the fact that the family
$\{\cos (k\theta ): k\in {\mathbb N}\}$ is independent, we conclude that
$b_{0}=0$ and
$L_{k}:=\sum \nolimits _{\sigma \in \Sigma } b_{k\sigma}\sigma =0$ for
$k=1,\dots ,n$. Further,
\begin{align*}
0 =&f(q)-f(q^{-1})= \sum _{k=1}^{n}\sum _{\sigma \in \Sigma } b_{k
\sigma} (q^{k}-q^{-k})\sigma
\\
=&2 \, \sum _{k=1}^{n}\sum _{\sigma \in \Sigma } b_{k\sigma} (\tau
\sin (k\theta ))\sigma = 2\, \sum _{k=1}^{n}\left [\sum _{\sigma \in
\Sigma } b_{k\sigma}\tau \sigma \right ]\sin (k\theta ).
\end{align*}
Since $\{\sin (k\theta ); k\in {\mathbb N}\}$ is independent, we obtain
$L_{k,\tau}:=\sum \nolimits _{\sigma \in \Sigma } b_{k\sigma}\tau
\sigma =0$ for $k=1,\dots ,n$.

Let $\tau ,\sigma \in \{\mathrm{i}, \mathrm{j}, \mathrm{k}\}$. Then
$\tau \sigma \tau =1$ for $\tau \ne \sigma $ and
$\tau \sigma \tau =-1$ for $\tau =\sigma $. Using this we compute
\begin{align*}
0= L_{k} -L_{k,\tau}\tau = 2b_{k,1} +2 b_{k,\tau} \tau +\sum _{
\sigma \neq 1, \tau} b_{k,\sigma} (\sigma -\tau \sigma \tau ) =2 b_{k,1}+
2b_{k\tau}\tau ,
\end{align*}
so that $b_{k\tau}=b_{k,1}\tau $. Therefore,
\begin{equation*}
0=L_{k}=b_{k,1}+b_{k \mathrm{i}}\mathrm{i}+b_{k \mathrm{j}}\mathrm{j}+b_{k
\mathrm{k}}\mathrm{k}=b_{k,1}(1+\mathrm{i}^{2}+\mathrm{j}^{2}+k^{2})=-2b_{k,1}.
\end{equation*}
Thus $b_{k,1}=0$ and hence $b_{k \sigma}=0$ for all $k=1,\dots ,n$ and
$\sigma \in \Sigma $.
\end{proof}

\begin{theorem}%
\label{propcanonicalrep}
The set
\begin{align*}
\mathcal{B} &= \{(q^{*})^{k} q^{n} \sigma :(k,n,\sigma ) \in
\mathcal{N}\}
\\
& \equiv \{(q^{*})^{k}q^{n}\sigma : k,n\in \mathbb{N}_{0}, n>k ,
\sigma \in \Sigma \} \cup \{(q^{*})^{n}q^{n}: n\in \mathbb{N}_{0}\}
\end{align*}
is a left $\mathbb{H}$-module basis of  $\mathbb{H}^{0}[q^{*}, q]$. That
is, each $f\in \mathbb{H}^{0}[q^{*}, q]$ can be written as a finite sum
\begin{align}
\label{canonicalrep}
f=\sum _{\substack{(k,n,\sigma )\in \mathcal{N}}} ~ a_{kn\sigma} (q^{*})^{k}
q^{n} \sigma ,\quad \textrm{with}~~~\, a_{kn\sigma} \in \mathbb{H},
\end{align}
and the coefficients $a_{kn\sigma} $ are uniquely determined by $f$ for
all $(k, n, \sigma )\in \mathcal{N}$.
\end{theorem}
\begin{proof}
First we verify that each $f\in \mathbb{H}^{0}[q^{*}, q]$ is of the form
(\ref{canonicalrep}). Let $f$ be as in Definition~\ref{defih0}. Since each
quaternion is a real combination of
$1,\mathrm{i}, \mathrm{j},\mathrm{k}$ and real numbers commute with quaternions,
$f$ is a finite sum of terms $a (q^{*})^{k} q^{n} \sigma $, where
$a\in \mathbb{H}$, $\sigma \in \Sigma $, and $k,n\in \mathbb{N}_{0}$. The
conjugation formula (\ref{conjugationformula}), applied to
$\sigma ^{*} (q^{*})^{n} q^{k}$, gives
\begin{equation*}
(q^{*})^{k} q^{n} \sigma = -\frac{1}{2}[\, \sigma ^{*} (q^{*})^{n} q^{k}
+ \mathrm{i}(\sigma ^{*} (q^{*})^{n} q^{k})\mathrm{i}+ \mathrm{j}(
\sigma ^{*} (q^{*})^{n} q^{k})\mathrm{j}+ \mathrm{k}(\sigma ^{*} (q^{*})^{n}
q^{k})\mathrm{k}\, ]
\end{equation*}
If $k>n$, it follows from this identity that
$(q^{*})^{k} q^{n} \sigma $ is a sum of terms
$a(q^{*})^{l} q^{m}\tau $ with $m>l$, $\tau \in \sigma $ and
$a\in \mathbb{H}$. If $k=n$, then
$(q^{*})^{n}q^{n}\sigma =\sigma (q^{*})^{n} q^{n}$. Thus in both cases,
$f$ is a finite sum of elements $a (q^{*})^{k} q^{n} \sigma $, where
$(k,n,\sigma )\in \mathcal{N}$ and $a\in \mathbb{H}$, that is, $f$ is of
the form (\ref{canonicalrep}).

Next we prove the uniqueness assertion. It suffices to show that
$f(q)=0$ for all $q\in \mathbb{H}$ implies that all coefficients
$a_{kn\sigma}$ vanish. We write $q=r\omega $, where $r=|q|$ and
$\omega $ is in the quaternionic sphere
$\mathbb{S}=\{q\in \mathbb{H}: |q|=1\}$. From
$\omega ^{*} = \omega ^{-1}$ for $\omega \in \mathbb{S}$ we obtain
$(q^{*})^{k} q^{n} = r^{k+n} \omega ^{n-k}$. Set $m := n + k$. Then
\begin{align}
\label{fq0romega}
f(q)=\sum _{m\geq 0} r^{m} \left [\sum _{
\substack{(k,n,\sigma )\in \mathcal{N},k+n=m}} a_{kn\sigma} \omega ^{n-k}
\sigma \right ] =0
\end{align}
for all $q\in \mathbb{H}$. Let $b_{m}$ denote the coefficient in squared
brackets. First we note that $b_{m}=0$ for all $m\in \mathbb{N}$. Indeed,
if there would be a non-zero $b_{m}$ with $m\in \mathbb{N}$, let $d$ be
the largest $m$ such that $b_{m}\neq 0$. Considering the norm of
$r^{d}b_{d}$ for sufficiently large $r>0$ in (\ref{fq0romega}) leads to
contradiction. Thus $b_{m}=0$ for $m\geq 1$. Let us fix
$m\in \mathbb{N}$.

Set $j:=n-k$. Then we have $k=(m-j)/2, n=(m+j)/2$ and
\begin{equation}
\label{bmzero}
b_{m}=\sum _{j,\sigma ;
\substack{((m-j)/2, (m+j)/2,\sigma )\in \mathcal{N}}} a_{(m-j)/2,(m+j)/2,
\sigma}\, \omega ^{j} \sigma = 0,
\end{equation}
for all $\omega \in \mathbb{S}$. (Note that the sum in (\ref{bmzero}) is
over all $j$ and $\sigma $ such that
$ ((m-j)/2, (m+j)/2,\sigma )\in \mathcal{N}$.) Therefore, it follows from
Lemma~\ref{fqzerolemma} that the coefficients of all expressions
$\omega ^{j}\sigma $ for $j\geq 1$ and the coefficient of
$\omega ^{0} \cdot 1$ are zero. That is,
$a_{(m-j)/2,(m+j)/2,\sigma}=0$ for all
$ ( (m-j)/2, (m+j)/2,\sigma )\in \mathcal{N}$ with $j\geq 0$ and
$a_{m/2,m/2,1}=0$ provided that $((m/2,m/2,1)\in \mathcal{N}$.

Finally, we consider the case $m=0$. Since $b_{m}=0$ for
$m\in \mathbb{N}$, we have $f(q)=b_{0}=a_{0,0,1}=0$. Thus we have shown
that all coefficients $ a_{kn\sigma}$ vanish, which completes the proof
of the uniqueness assertion.
\end{proof}

\begin{definition}
\label{defn4.4}
(\ref{canonicalrep}) is called the \emph{canonical representation} of
$f\in \mathbb{H}^{0}[q^{*}, q]$.
\end{definition}

The space $ \mathbb{H}^{0}[q^{*}, q]$ is not closed under multiplication.
The subset
\begin{equation}
\mathbb{H}^{0}_{\mathbb{R}}[q^{*}, q] = \Big\{ \sum \nolimits _{k,n} a_{kn}
\, (q^{*})^{k} q^{n} : \, a_{kn}\in \mathbb{R}\Big\}
\label{eq24}
\end{equation}
of $\mathbb{H}^{0}[q^{*}, q] $ is a real vector space which is invariant
under the involution. Since $q$ and $q^{*}$ commute, it is also invariant
under multiplication. Thus $\mathbb{H}^{0}_{\mathbb{R}}[q^{*}, q]$ is a
commutative unital real $*$-algebra.

The hermitian parts
\begin{align*}
\mathbb{H}^{0}_{\mathbb{R}}[q^{*}, q]_{\mathrm{her}} & =\{ f\in
\mathbb{H}^{0}_{\mathbb{R}}[q^{*}, q]: \,f=f^{*}\},
\\
\mathbb{H}^{0}[q^{*}, q]_{\mathrm{her}} & =\{ f\in \mathbb{H}^{0}[q^{*},
q]: \,f=f^{*}\}
\end{align*}
are explicitly described in the following theorem.
\begin{theorem}%
\label{realpartH_L}
\begin{align}
\label{herpart}
\mathbb{H}^{0}_{\mathbb{R}}[q^{*}, q]_{\mathrm{her}} =& ~\mathbb{R}[x_{0},x_{1}^{2}+x_{2}^{2}+x_{3}^{2}],
\\
\mathbb{H}^{0}[q^{*}, q]_{\mathrm{her}} =& ~\mathbb{R}[x_{0},x_{1}^{2}+x_{2}^{2}+x_{3}^{2}]
+ x_{1}\, \mathbb{R}[x_{0},x_{1}^{2}+x_{2}^{2}+x_{3}^{2}]
\label{herpart1}
\\
&+x_{2}\, \mathbb{R}[x_{0},x_{1}^{2}+x_{2}^{2}+x_{3}^{2}]+ x_{3}\,
\mathbb{R}[x_{0},x_{1}^{2}+x_{2}^{2}+x_{3}^{2}].
\nonumber
\end{align}
\end{theorem}

\begin{proof}
We write $q=x_{0}+x_{I}$, where
$x_{I}=\mathrm{i}x_{1}+\mathrm{j}x_{2}+\mathrm{k}x_{3}$. Then
$q^{*}=x_{0}-x_{I}$.

Since $x_{0}$ and $x_{I}$ commute, the binomial theorem applies and we
obtain
\begin{align*}
(q^{*})^{k}q^{n} &= (x_{0}-x_{I})^{k}(x_{0}+x_{I})^{n}
\\
&=\left (\sum _{j=0}^{k}\binom{k}{j} x_{0}^{k-j}(-x_{I})^{j}\right )
\left (\sum _{\ell =0}^{n}\binom{n}{\ell} x_{0}^{n-\ell}x_{I}^{\ell}
\right )
\\
&= \sum _{j=0}^{k} \sum _{\ell =0}^{n} \binom{k}{j}\binom{n}{\ell} (-1)^{j}
x_{0}^{k+n-j-\ell} x_{I}^{j+\ell}.
\end{align*}

For $r\in \mathbb{N}$, we have
\begin{align*}
x_{I}^{2r}&=(-1)^{r}\|x_{I}\|^{2r} = (-1)^{r}(x_{1}^{2}+x_{2}^{2}+x_{3}^{2})^{r},
\\
x_{I}^{2r+1}&=(-1)^{r}\|x_{I}\|^{2r} x_{I} = (-1)^{r}(x_{1}^{2}+x_{2}^{2}+x_{3}^{2})^{r}
x_{I}.
\end{align*}
Hence, $\mathrm{Re}((q^{*})^{k}q^{n})$ is a real combination of terms
$x_{0}^{k+n-j-\ell}x_{I}^{j+\ell}$ with $j+\ell $ even, and
$\mathrm{Im}((q^{*})^{k}q^{n})$ is a real combination of terms
$x_{0}^{k+n-j-\ell}x_{I}^{j+\ell}$ with $j+\ell $ odd. Therefore,
\begin{align}
\label{re}
\mathrm{Re}((q^{*})^{k}q^{n}) & \in \mathbb{R}[x_{0},x_{I}^{2}]
\subseteq \mathbb{R}[x_{0},x_{1}^{2}+x_{2}^{2}+x_{3}^{2}],
\\
\mathrm{Im}((q^{*})^{k}q^{n}) & \in x_{I} \, \mathbb{R}[x_{0},x_{I}^{2}]
\subseteq x_{I} \, \mathbb{R}[x_{0},x_{1}^{2}+x_{2}^{2}+x_{3}^{2}].
\label{im}
\end{align}
In particular, from (\ref{EherRe}) and (\ref{re}) we obtain
\begin{align}
\label{hrinclusion}
\mathbb{H}^{0}_{\mathbb R}[q^{*}, q]_{\mathrm{her}} \subseteq
\mathbb{R}[x_{0},x_{1}^{2}+x_{2}^{2}+x_{3}^{2}].
\end{align}

Now we prove the converse inclusion in (\ref{hrinclusion}). Let
$p,s\in \mathbb{N}$. Then
\begin{align}
x_{0}^{p} & = \frac{1}{2^{p}}\, (q+q^{*})^{p} = \frac{1}{2^{p}} \sum _{k=0}^{p}
\binom{p}{k} q^{k} (q^{*})^{p-k},
\label{x0}
\\
(x_{1}^{2}+x_{2}^{2}+x_{3}^{2})^{s} & = (q^{*}q - x_{0}^{2})^{s} =
\sum _{j=0}^{s} \binom{s}{j} (-1)^{s-j} (q^{*})^{j} q^{j} (x_{0}^{2})^{s-j}.
\label{xsum}
\end{align}
(Note that the binomial theorem applies, because $q,q^{*}$ and
$q^{*}q, x_{0}^{2}$ commute.) By (\ref{x0}), $(x_{0}^{2})^{s-j}$ is a real
combination of terms $(q^{*})^{k}q^{n}$. Therefore, since
$\mathbb{H}^{0}_{\mathbb R}[q^{*}, q]$ is a commutative algebra, it follows
from (\ref{x0}) and (\ref{xsum}) that the product
$x_{0}^{p}(x_{1}^{2}+x_{2}^{2}+x_{3}^{2})^{s}$ is a real combination of
terms $(q^{*})^{k}q^{n}$, so it is contained in
$ \mathbb{H}^{0}_{\mathbb{R}}[q^{*}, q]$. Because it is real-valued on
$\mathbb{H}$,
$x_{0}^{p}(x_{1}^{2}+x_{2}^{2}+x_{3}^{2})^{s} \in \mathbb{H}^{0}_{
\mathbb R}[q^{*}, q]_{\mathrm{her}}$. Hence
$ \mathbb{R}[x_{0},x_{1}^{2}+x_{2}^{2}+x_{3}^{2}]\subseteq \mathbb{H}^{0}_{
\mathbb R}[q^{*}, q]_{\mathrm{her}} $. Combined with (\ref{hrinclusion}),
this proves (\ref{herpart}).

Now we turn to the proof of (\ref{herpart1}). For short, we denote the
set at the right-hand side of (\ref{herpart1}) by $\mathcal{R}$. Then the
assertion (\ref{herpart1}) means that
$ \mathbb{H}^{0}_{\mathbb R}[q^{*}, q]_{\mathrm{her}}=\mathcal{R}$.

First we show that
$ \mathbb{H}^{0}[q^{*}, q]_{\mathrm{her}}\subseteq \mathcal{R}$. By (\ref{EherRe})
it suffices to prove that
\begin{align}
\label{rekn}
\mathrm{Re} (\tau [ (q^{*})^{k}q^{n} ]\sigma ) \in \mathcal{R}\quad
\text{ for all} ~~~ k,n\in \mathbb{N}~~~\text{and}~~\tau , \sigma \in
\{1,i,j,k\}.
\end{align}
Since
$\mathrm{Re}[(q^{*})^{k}q^{n}] \in \mathbb{R}[x_{0},x_{1}^{2}+x_{2}^{2}+x_{3}^{2}]$
as noted above, we conclude that
\begin{align}
\label{reqkn}
\mathrm{Re} (\tau \, \mathrm{Re} [ (q^{*})^{k}q^{n} ]\,\sigma )=
\mathrm{Re} [ (q^{*})^{k}q^{n} ]\, \cdot \, \mathrm{Re} (\tau \sigma )
\in \mathcal{R}.
\end{align}
For the imaginary part we have
$\mathrm{Im}[(q^{*})^{k}q^{n}] \in x_{I} \, \mathbb{R}[x_{0},x_{1}^{2}+x_{2}^{2}+x_{3}^{2}]$
by (\ref{im}). Hence
\begin{align}
\label{imqkn1}
\mathrm{Re} (\tau \, \mathrm{Im} \, [ (q^{*})^{k}q^{n} ]\, \sigma )=
\mathbb{R}[x_{0},x_{1}^{2}+x_{2}^{2}+x_{3}^{2}]\, \cdot \,
\mathrm{Re} (\tau x_{I} \sigma ).
\end{align}
Since
$\tau x_{I} \sigma =\tau i \sigma x_{1}+ \tau j\sigma x_{2}+ \tau k
\sigma x_{3}$, $\mathrm{Re} (\tau x_{I} \sigma )$ is a real linear combination
of $x_{1},x_{2},x_{3}$. Thus it follows from (\ref{imqkn1}) that
\begin{align}
\label{imqkn}
\mathrm{Re} (\tau \, \mathrm{Im} [ (q^{*})^{k}q^{n} ]\,\sigma )\in
\mathcal{R}.
\end{align}
Clearly, (\ref{reqkn}) and (\ref{imqkn}) imply (\ref{rekn}) which proves
that $ \mathbb{H}^{0}[q^{*}, q]_{\mathrm{her}}\subseteq \mathcal{R}$.

Now we verify the opposite inclusion
$\mathcal{R}\subseteq \mathbb{H}^{0}[q^{*}, q]_{\mathrm{her}}$. Suppose
that $f\in \mathbb{R}[x_{0},x_{1}^{2}+x_{2}^{2}+x_{3}^{2}]$. Set
$g:=\frac{1}{2}(q^{*}\mathrm{i}-\mathrm{i}q) f$. Since
$f\in \mathbb{H}_{\mathbb{R}}[q^{*}, q]_{\mathrm{her}}$ by (\ref{herpart}),
we have $g\in \mathbb{H}^{0}[q^{*}, q]$ and
$\mathrm{Re} ~ g= x_{1} f$. This shows that
\begin{equation*}
x_{1} \mathbb{R}[x_{0},x_{1}^{2}+x_{2}^{2}+x_{3}^{2}] \subseteq
\mathbb{H}^{0}[q^{*}, q]_{\mathrm{her}}.
\end{equation*}
Replacing $x_{1}= \frac{1}{2}(q^{*}\mathrm{i}-\mathrm{i}q)$ by
$x_{2}= \frac{1}{2}(q^{*}\mathrm{j}- \mathrm{j}q) $ and
$x_{3}= \frac{1}{2}(q^{*}\mathrm{k}-\mathrm{k}q) $, respectively, the same
reasoning yields
\begin{equation*}
x_{2} \mathbb{R}[x_{0},x_{1}^{2}+x_{2}^{2}+x_{3}^{2}] \subseteq
\mathbb{H}^{0}[q^{*}, q]_{\mathrm{her}}, ~~~ x_{3} \mathbb{R}[x_{0},x_{1}^{2}+x_{2}^{2}+x_{3}^{2}]
\subseteq \mathbb{H}^{0}[q^{*}, q]_{\mathrm{her}}.
\end{equation*}
Therefore,
$\mathcal{R}\subseteq \mathbb{H}^{0}[q^{*}, q]_{\mathrm{her}}$. Thus, together
with the preceding paragraph we have shown that equation (\ref{herpart1})
holds. This completes the proof of Theorem~\ref{realpartH_L}.
\end{proof}

\begin{remark}
\label{rem4.6}
1. The preceding computations for
$ \mathbb{H}^{0}[q^{*}, q]_{\mathrm{her}}$ go also through in the complex
case, but then for $z=x_{0}+\mathrm{i}x_{1}$ we have
\begin{equation*}
\mathbb{R}[x_{0},x_{1}^{2}]+ x_{1}\mathbb{R}[x_{0},x_{1}^{2}]=
\mathbb{R}[x_{0},x_{1}].
\end{equation*}
That is, in the complex case we obtain the whole polynomial algebra
$\mathbb{R}[x_{0},x_{1}]$.

2. The hermitian part
$\mathbb{H}_{\mathbb{R}}[q^{*}, q]_{\mathrm{her}} =\mathbb{R}[x_{0},x_{1}^{2}+x_{2}^{2}+x_{3}^{2}]$
is an important ingredient for the study of the quaternionic moment problem.
But the polynomial space $\mathbb{H}_{\mathbb R}[q^{*}, q]$ is too small
to treat the quaternionic moment problem with our methods. Since
$\mathbb{H}_{\mathbb R}[q^{*}, q]$ is not an $\mathbb{H}$-linear space,
the quaternionic Choquet Theorem~\ref{quatchoquet} does not apply to
$\mathbb{H}_{\mathbb R}[q^{*}, q]$.
\end{remark}

\subsection{The intermediate quaternionic polynomial spaces}
\label{sec4.2}

\begin{definition}%
\label{defihk}
For $k\in \mathbb{N}$ we define
\begin{equation*}
\begin{array}{r@{\quad}l}
\mathbb{H}^{k}[q^{*}, q] =& \mathrm{span}_{\mathbb{H}} \, \{(q^{*})^{m_{0}}
q^{n_{0}} \sigma _{1}(q^{*})^{m_{1}} q^{n_{1}}\cdots (q^{*})^{m_{k-1}}
q^{n_{k-1}}\sigma _{k}(q^{*})^{m_{k}} q^{n_{k}}:
\\
& \quad \quad \quad \, \, m_{0},n_{0},\dots ,m_{k},n_{k}\in
\mathbb{N}_{0}\,, \sigma _{1},\dots ,\,\sigma _{k}\in \Sigma \}.
\end{array}
\end{equation*}
Recall that $\mathrm{{span}_{\mathbb{H}}}$ means the spanned two-sided
$\mathbb{H}$-linear space.
\end{definition}
In particular,
\begin{equation*}
\begin{array}{l@{\quad}l}
\mathbb{H}^{1}[q^{*}, q] = & \Big\{ \sum _{j,k,l,m,n} a_{jklmn}\, (q^{*})^{k}
q^{l} \sigma (q^{*})^{m} q^{n}\, b_{jklmn}:
\\
& \quad \, j,k,l,m, n \in \mathbb{N}_{0}, \, a_{jklmn}, b_{jklmn}\in
\mathbb{H},\, \sigma \in \Sigma \Big\}.
\end{array}
\end{equation*}

\begin{proposition}
\label{prop4.8}
For $k\in \mathbb{N}$ the following multiplicativity property holds:
\begin{align}
\label{multprop}
\mathbb{H}^{k}[q^{*}, q] ={\mathrm{span}}_{\mathbb{R}}\, \{\mathbb{H}^{k-1}[q^{*},
q] \cdot \mathbb{H}^{0}[q^{*}, q]\}.
\end{align}
(More precisely, $\mathbb{H}^{k}[q^{*}, q]$ is the set of finite sums of
products $f\cdot g$, where $f\in \mathbb{H}^{k-1}[q^{*}, q]$ and
$g\in \mathbb{H}^{0}[q^{*}, q]$.)
\end{proposition}
\begin{proof}
Each element of $ \mathbb{H}^{k}[q^{*}, q] $ is a finite sum of terms
\begin{equation*}
[a(q^{*})^{m_{0}} q^{n_{0}} \sigma _{1}(q^{*})^{m_{1}} q^{n_{1}}
\cdots (q^{*})^{m_{k-1}} q^{n_{k-1}}]\cdot [\sigma _{k}(q^{*})^{m_{k}}
q^{n_{k}}b],
\end{equation*}with $a,b\in \mathbb{H}$. Such a summand is contained in
$ \mathbb{H}^{k-1}[q^{*}, q] \cdot \mathbb{H}^{0}[q^{*}, q]$.

Conversely, elements of
$ \mathbb{H}^{k-1}[q^{*}, q] \cdot \mathbb{H}^{0}[q^{*}, q]$ are finite
sums of
\begin{equation*}
[a(q^{*})^{m_{0}} q^{n_{0}} \sigma _{1}(q^{*})^{m_{1}} q^{n_{1}}
\cdots (q^{*})^{m_{k-1}} q^{n_{k-1}}b]\cdot [c (q^{*})^{m_{k}} q^{n_{k}}d],
\end{equation*}
with $a,b,c,d\in \mathbb{H}$. Writing $b$ and $c$ as real combinations
of $1,\mathrm{i},\mathrm{j},\mathrm{k}$ it follows that such a polynomial
belongs to $\mathbb{H}^{k}[q^{*}, q]$.
\end{proof}
\begin{lemma}%
\label{frep}
Let $k\in \mathbb{N}$. Each $f\in \mathbb{H}^{k}[q^{*},q]$ has a unique
representation
\begin{align}
\label{frepresentation}
f=f_{0}+\mathrm{i}f_{1}+\mathrm{j}f_{2}+\mathrm{k}f_{3},
\end{align}
with
$f_{0}, f_{1},f_{2},f_{3}\in \mathbb{H}^{k}[q^{*},q]_{\mathrm{her}}$ and
we have
\begin{align}
\label{dechk}
\mathbb{H}^{k}[q^{*},q]=\mathbb{H}^{k}[q^{*},q]_{\mathrm{her}}+
\mathrm{i}\, \mathbb{H}^{k}[q^{*},q]_{\mathrm{her}}+\mathrm{j}\,
\mathbb{H}^{k}[q^{*},q]_{\mathrm{her}}+\mathrm{k}\, \mathbb{H}^{k}[q^{*},q]_{
\mathrm{her}}.
\end{align}
\end{lemma}
\begin{proof}
For $f\in \mathbb{H}^{k}[q^{*},q]$ we define
\begin{align*}
f_{0}=\frac{f+f^{*}}{2}, ~ f_{1}=
\frac{-f\mathrm{i}+\mathrm{i}f^{*}}{2}, ~f_{2}=
\frac{-f\mathrm{j}+ \mathrm{j}f^{*}}{2},~ f_{3}=
\frac{-f\mathrm{k}+ \mathrm{k}f^{*}}{2}.
\end{align*}
Clearly, each $f_{j}$ is in $\mathbb{H}^{k}[q^{*},q]_{\mathrm{her}}$. A
simple computation using the conjugation formula, applied to
$f(q)^{*}$, shows that equation (\ref{frepresentation}) holds.

Conversely, inserting (\ref{frepresentation}) into
$\frac{f+f^{*}}{2}$, $\frac{-f\mathrm{i}+\mathrm{i}f^{*}}{2}$, $\frac{-f\mathrm{j}+ \mathrm{j}f^{*}}{2}$,  $\frac{-f\mathrm{k}+ \mathrm{k}f^{*}}{2}$  we obtain
$f_{0}, f_{1}, f_{2}, f_{3}$, respectively. This proves the uniqueness
of the representation (\ref{frepresentation}).

Obviously, (\ref{frepresentation}) implies (\ref{dechk}).
\end{proof}
\begin{lemma}
\label{lem4.10}
For $k\in \mathbb{N}$,
\begin{align}
\label{multher}
\mathbb{H}^{k}[q^{*},q]_{\mathrm{her}}=\mathrm{span}_{\mathbb{R}}\,
\{\mathbb{H}^{k-1}[q^{*},q]_{\mathrm{her}}\cdot \mathbb{H}^{0}[q^{*},q]_{
\mathrm{her}} \, \}
\end{align}
\end{lemma}
\begin{proof}
From the multiplicativity property (\ref{multprop}) and (\ref{EherRe})
we obtain
\begin{align}
\label{producther}
\mathbb{H}^{k}[q^{*},q]_{\mathrm{her}}={\mathrm{span}}_{\mathbb{R}}\,
\big\{ \mathrm{Re}(f\cdot g): f\in \mathbb{H}^{k-1}[q^{*},q], g\in
\mathbb{H}^{0}[q^{*},q] \, \big\}.
\end{align}
Let $ f\in \mathbb{H}^{k-1}[q^{*},q] $ and
$ g\in \mathbb{H}^{0}[q^{*},q] $. Using Lemma~\ref{frep} we can write
$f=f_{0}+\mathrm{i}f_{1}+\mathrm{j}f_{2}+\mathrm{k}f_{3}$ and
$g=g_{0}+\mathrm{i}g_{1}+\mathrm{j}g_{2}+\mathrm{k}g_{3}$ with
$ f_{j}\in \mathbb{H}^{k-1}[q^{*},q]_{\mathrm{her}} $ and
$ g_{j}\in \mathbb{H}^{0}[q^{*},q]_{\mathrm{her}} $. Then
\begin{equation*}
\mathrm{Re}(f\cdot g)= f_{0}g_{0}-f_{1}g_{1}-f_{2}g_{2}-f_{3}g_{3}
\in \mathrm{span}_{\mathbb{R}}\, \{\mathbb{H}^{k-1}[q^{*},q]_{
\mathrm{her}}\cdot \mathbb{H}^{0}[q^{*},q]_{\mathrm{her}} \, \}.
\end{equation*}
By (\ref{producther}) this proves that
$ \mathbb{H}^{k}[q^{*},q]_{\mathrm{her}}\subseteq \mathrm{span}_{
\mathbb{R}}\, \{\mathbb{H}^{k-1}[q^{*},q]_{\mathrm{her}}\cdot
\mathbb{H}^{0}[q^{*},q]_{\mathrm{her}} \}$.

To verify the opposite inclusion let
$ f\in \mathbb{H}^{k-1}[q^{*},q]_{\mathrm{her}}$ and
$ g\in \mathbb{H}^{0}[q^{*},q]_{\mathrm{her}}$. Then
$fg=\mathrm{Re}[ (f+\mathrm{i}\, 0 +\mathrm{j}\, 0+\mathrm{k}\, 0)(g+
\mathrm{i}\,0+\mathrm{j}\, 0+\mathrm{k}\,0)]\in \mathbb{H}^{k}[q^{*},q]_{
\mathrm{her}}$. This completes the proof of the equality (\ref{producther}).
\end{proof}

For
$\alpha =(\alpha _{1},\alpha _{2},\alpha _{3})\in \mathbb{N}_{0}^{3}$ we
define $|\alpha |:=\alpha _{1}+\alpha _{2}+\alpha _{3}$ and
\begin{align}
x^{\alpha }= x_{1}^{\alpha _{1}}\, x_{2}^{\alpha _{2}}\, x_{3}^{\alpha _{3}},
\label{eq41}
\end{align}
where $x_{j}^{0}:=1$. Note that $|\alpha |$ is the degree of the monomial
$x^{\alpha}$. The following theorem describes the hermitian part of the
polynomial space $\mathbb{H}^{k}[q^{*},q]$.
\begin{theorem}%
\label{herhk}
For $k\in \mathbb{N}_{0}$, we have
\begin{align}
\label{hhherk}
\mathbb{H}^{k}[q^{*},q]_{\mathrm{her}}=\sum _{\alpha \in \mathbb{N}_{0}^{3},
|\alpha |\leq k+1} x^{\alpha}\, \mathbb{R}[x_{0},x_{1}^{2}+x_{2}^{2}+x_{3}^{3}].
\end{align}
\end{theorem}
\begin{proof}
The proof is by induction on $k$. For this proof we abbreviate
$\mathcal{R}:= \mathbb{R}[x_{0},x_{1}^{2}+x_{2}^{2}+x_{3}^{3}]$. For
$k=0$ the assertion
\begin{align}
\label{case0}
\mathbb{H}^{0}[q^{*},q]_{\mathrm{her}}=\mathcal{R}+x_{1}\mathcal{R}+x_{2}
\mathcal{R}+x_{3}\mathcal{R}
\end{align}
was stated as Theorem~\ref{realpartH_L}.

Assume that (\ref{hhherk}) holds for $k$. Then
\begin{align}
\label{casek}
\mathbb{H}^{k}[q^{*},q]_{\mathrm{her}}=\sum _{\alpha \in \mathbb{N}_{0}^{3},
|\alpha |\leq k+1} x^{\alpha}\, \mathcal{R}.
\end{align}
Multiplying (\ref{casek}) by (\ref{case0}) and using (\ref{multher}) for
$k+1$ gives the assertion (\ref{hhherk}) for $k+1$.
\end{proof}

\subsection{The full quaternionic polynomial space}
\label{sec4.3}

An immediate consequence of Definitions~\ref{defih0} and \ref{defihk} is
\begin{align}
\label{hkinclusion}
\mathbb{H}^{k}[q^{*},q]\subseteq \mathbb{H}^{k+1}[q^{*},q], \quad k
\in \mathbb{N}_{0}.
\end{align}
\begin{definition}%
\label{defihi}
\begin{equation*}
\begin{array}{r@{\quad}l}
\mathbb{H}^{\infty}[q^{*},q]&= \, \bigcup _{k=0}^{\infty}\mathbb{H}^{k}[q^{*},q]
\\
&= \, \mathrm{span}_{\mathbb{H}} \, \{ (q^{*})^{m_{0}} q^{n_{0}}
\sigma _{1}(q^{*})^{m_{1}} q^{n_{1}}\cdots (q^{*})^{m_{k-1}} q^{n_{k-1}}
\sigma _{k}(q^{*})^{m_{k}} q^{n_{k}}:
\\
& \quad \quad \quad \quad \quad m_{0},n_{0},\dots ,m_{k},n_{k}\in
\mathbb{N}_{0}\,, \sigma _{1},\dots ,\sigma _{k}\in \Sigma ,\, k\in
\mathbb{N}_{0} \}.
\end{array}
\end{equation*}
\end{definition}

From the second equality in Definition~\ref{defihi} it follows that
$\mathbb{H}^{\infty}[q^{*},q]$ is closed under multiplication and invariant
under the involution of $C(\mathbb{H};\mathbb{H})$. Hence
$\mathbb{H}^{\infty}[q^{*},q]$ is a unital $*$-algebra over the quaternions
\cite{Schmudgen2020}.

Note that for $k\in \mathbb{N}_{0}$ the polynomial space
$\mathbb{H}^{k}[q^{*},q]$ is $*$-invariant, but it is not closed under
multiplication.
\begin{theorem}%
\label{herinfty}
\begin{align}
\label{hermhinfty}
\mathbb{H}^{\infty}[q^{*},q]_{\mathrm{her}}=\mathbb{R}[x_{0},x_{1},x_{2},x_{3}].
\end{align}
\end{theorem}
\begin{proof}
Inserting the description of
$\mathbb{H}^{k}[q^{*},q]_{\mathrm{her}}$ from Theorem~\ref{herhk} into
the first equality in Definition~\ref{defihi} yields (\ref{hermhinfty}).
\end{proof}
\begin{remark}
\label{rem4.14}
In \cite[Subsection 1.4]{AP21}\footnote{The authors would like to thank
the referee for pointing out this reference.} it is shown that
$\mathbb{H}^{\infty}[q^{*},q]$ is equal to
\begin{equation*}
\mathbb{R}[x_{0},x_{1},x_{2},x_{3}]+\mathbb{R}[x_{0},x_{1},x_{2},x_{3}]
\, \mathrm{i}+\mathbb{R}[x_{0},x_{1},x_{2},x_{3}]\, \mathrm{j}+
\mathbb{R}[x_{0},x_{1},x_{2},x_{3}]\, \mathrm{k}.
\end{equation*}
This implies the equality (\ref{hermhinfty}).
\end{remark}

\subsection{Adapted space property of $\mathbb{H}^{k}[q^{*},q]_{\mathrm{her}}$}
\label{sec4.4}

For the applications to the moment problem given later we need the adapted
space property.
\begin{proposition}%
\label{adaptedH_L}
Suppose that $k\in \mathbb{N}_{0}\cup \{\infty \}$. For any closed subset
$K $ of $ \mathbb{R}^{4}\cong \mathbb{H}$, the space
${\mathbb{H}^{k}[q^{*},q]_{\mathrm{her}}\lceil K}$ of restrictions of hermitian
quaternion polynomials to $K$ forms an adapted linear subspace of
$C(K; \mathbb{R})$ according to Definition~\ref{adapted}.
\end{proposition}

\begin{proof}
We have to verify the three conditions in Definition~\ref{adapted} for
adapted spaces:

\text{(i) Difference property:} Let
$f \in \mathbb{H}^{k}[q^{*}, q]_{\mathrm{her}}$. Since $f$ has only finitely
many terms, there exist $M > 0$ and $d\in \mathbb{N}$ such that
$|f(q)| \leq M(1 + \|q\|^{2d})$ for all $q \in \mathbb{H}$. Define
$f_{1}(q) = f(q) + M(q^{*}q + 1)^{d}$ and
$f_{2}(q) = M(q^{*}q + 1)^{d}$. Then we have
$f_{1},f_{2} \in (\mathbb{H}^{k}[q^{*}, q]_{\mathrm{her}}\lceil _{K})_{+}$
and $f=f_{1}-f_{2}$.

\text{(ii) Strict positivity:} The polynomial $P(q) = q^{*}q + 1$ satisfies
$P(q) > 0$ for all $q \in K$.

\text{(iii) Domination:} For
$f \in (\mathbb{H}^{k}[q^{*}, q]_{\mathrm{her}}\lceil _{K})_{+}$, define
$g(q) = f(q)(q^{*}q + 1)$. Then
$g\in (\mathbb{H}^{k}[q^{*}, q]_{\mathrm{her}}\lceil _{K})_{+}$ and
$g$ dominates $f$ on $K$.
\end{proof}

\section{The moment problem for basic quaternionic polynomials}
\label{mpbasic}

It is convenient to introduce the index set
\begin{align}
{\mathsf{{N}}} :=\{(m,n): m,n\in \mathbb{N}_{0}, n\geq m\}.
\label{eq47}
\end{align}
By a \emph{quaternionic sequence} we mean a sequence
$S=\{s_{(m,n)}\}_{(m,n) \in \mathsf{{N}}}$ of quaternions
$s_{(m,n)}\in \mathbb{H}$. Recall that
\begin{equation*}
\mathcal{N}=\{(k,n,\sigma ):k,n\in \mathbb{N}_{0}, n>k,\sigma \in
\Sigma \}\cup \{(n,n,1):n\in \mathbb{N}_{0}\}.
\end{equation*}
By Theorem~\ref{propcanonicalrep}, each
$f\in \mathbb{H}^{0}[q^{*}, q]$ has a canonical representation
\begin{align}
\label{canonicalrep1}
f=\sum _{\substack{(m,n,\sigma )\in \mathcal{N}}} ~ a_{(m,n,\sigma )} (q^{*})^{m}
q^{n} \sigma ,%
\end{align}
with coefficients $a_{(m,n,\sigma )}\in \mathbb{H}$ uniquely determined
by $f$.

Let $S=\{s_{(m,n)}\}_{(m,n) \in \mathsf{{N}}}$ be a quaternionic sequence. Since
the coefficients $a_{(m,n,\sigma )}$ are uniquely determined by $f$, there
is a well-defined (!) functional $L_{S}$ on
$\mathbb{H}^{0}[q^{*}, q]$ such that
\begin{equation}
\label{defintionriesz}
L_{S}(f) := \sum _{(m,n,\sigma )\in \mathcal{N}} a_{(m,n,\sigma )}s_{(m,n)}
\sigma ,
\end{equation}
where $f$ is given by (\ref{canonicalrep1}).
\begin{definition}
\label{defn5.1}
$L_{S}$ is called the \emph{Riesz functional} associated with the quaternionic
sequence $S$.
\end{definition}

From the definition it is clear that the functional $L_{S}$ is left
$\mathbb{H}$-linear.
\begin{lemma}%
\label{twosidedls}
$L_{S}$ is a right $\mathbb{H}$-linear functional on $ \mathbb{H}^{0}[q^{*}, q]$ if and only if $s_{(n,n)}\in \mathbb{R}$ for
all $n\in \mathbb{N}_{0}$.
\end{lemma}
\begin{proof}
Recall that the functional $L_{S}$ is right $\mathbb{H}$-linear means that
$L_{S}(f c)=L_{S}(f)c$ for all $f\in \mathbb{H}^{0}[q^{*}, q]$ and
$c\in \mathbb{H}$. Since $L_{S}$ is additive, it suffices to verify this
for a single summand $f=a (q^{*})^{m} q^{n} \sigma $, where
$a\in \mathbb{H}$, $(m,n,\sigma )\in \mathcal{N}$. We write
$c =c_{0}+c_{1}\mathrm{i}+c_{2}\mathrm{j}+ c_{3}\mathrm{k}$ with
$c_{0}, c_{1},c_{2},c_{3}\in \mathbb{R}$.

First we suppose that $m\neq n$. Then
\begin{align*}
f c =& ~ a (q^{*})^{m} q^{n} \sigma c
\\
= & ~ a c_{0} (q^{*})^{m} q^{n} \sigma + a c_{1} (q^{*})^{m} q^{n} (
\sigma \mathrm{i}) + a c_{2} (q^{*})^{m} q^{n} (\sigma \mathrm{j})+ a c_{3}
(q^{*})^{m} q^{n} (\sigma \mathrm{k}).
\end{align*}
Note that for arbitrary $\sigma , \tau \in \Sigma $, we have
$\sigma \tau \in \Sigma $. If $m\neq n$, this implies that
$(q^{*})^{m} q^{n} (\sigma \tau )$ belongs to the left module basis
$\mathcal{B}$ of $ \mathbb{H}^{0}[q^{*}, q]$. Therefore, by the definition
of $L_{S}$,
\begin{align*}
L_{S}(f c)= & ~ a c_{0} s_{(m,n)} \sigma + a c_{1}s_{(m,n)} \sigma
\mathrm{i}+ a c_{2} s_{(m,n)} \sigma \mathrm{j}+ a c_{3} s_{(m,n)}
\sigma \mathrm{k}
\\
=&~ a s_{(m,n)}\,\sigma (c_{0}+c_{1}\mathrm{i}+c_{2}\mathrm{j}+ c_{3}
\mathrm{k})=L_{S}(f) c .
\end{align*}

Now suppose that $m=n$. Then $\sigma =1$ by the definition of the index
set $\mathcal{N}$ and $fc=a(q^{*})^{n}q^{n} c=ac(q^{*})^{n}q^{n}$. Hence
$L_{S}(fc)=acs_{(n,n)}$ and $L_{S}(f)c=as_{(n,n)}c$. Therefore we have
$L_{S}(fc)=L_{S}(f)c$ for all $c\in \mathbb{H}$ if and only if
$s_{(n,n)}$ is real.
\end{proof}

Now we extend the quaternionic sequence
$S=\{s_{(m,n)}\}_{(m,n) \in \mathsf{{N}}}$ to indices
$(m,n)\in \mathbb{N}_{0}^{2}$ with $m>n$ by defining
\begin{align}
\label{mn}
s_{(m,n)}:=s_{(n.m)}^{*}\quad \text{for}\quad m>n.
\end{align}
Then, by (\ref{conjugationformula}),
\begin{equation}
\label{smnnm}
s_{(m,n)} = -\frac{1}{2}[ s_{(n,m)}+ \mathrm{i}s_{(n,m)} \mathrm{i}+
\mathrm{j}s_{(n,m)} \mathrm{j}+ \mathrm{k}s_{(n,m)} \mathrm{k}]
\quad \textrm{for}~~~ m>n.
\end{equation}
On the other hand, the conjugation formula (\ref{conjugationformula}) for
$ (q^{*})^{n} q^{m}$ yields
\begin{equation}
\label{qmstarqn}
(q^{*})^{m} q^{n} = -\frac{1}{2}[ (q^{*})^{n} q^{m} + \mathrm{i}( (q^{*})^{n}
q^{m})\mathrm{i}+ \mathrm{j}( (q^{*})^{n} q^{m})\mathrm{j}+
\mathrm{k}( (q^{*})^{n} q^{m})\mathrm{k}]
\end{equation}

Then, for all $m,n\in \mathbb{N}, m\neq n$, we have
\begin{align}
\label{smnsnmstar}
L_{S}((q^{*})^{m}q^{n})=s_{(m,n)}=s_{(n,m)}^{*}=L_{S}((q^{*})^{n}q^{m})^{*}.
\end{align}
Indeed, the first equality of (\ref{smnsnmstar}) follows from (\ref{defintionriesz}),
(\ref{smnnm}), (\ref{qmstarqn}) and the second equality follows from (\ref{mn}).
The last equality of (\ref{smnsnmstar}) follows by applying the adjoint
to the first equality.

\begin{lemma}%
\label{lsstarinvariant}
$L_{S}$ is $*$-invariant (that is, $L(f^{*})=L(f)^{*}$ for all
$f\in \mathbb{H}^{0}[q^{*}, q]$) if and only if
$s_{(n,n)}\in \mathbb{R}$ for $n\in \mathbb{N}_{0}$.
\end{lemma}
\begin{proof}
The only if assertion is clear, since $L(f^{*})=L(f)^{*}$ for
$f=(q^{*})^{n}q^{n}$ yields $s_{(n,n)}\in \mathbb{R}$.

Now we assume that $s_{(n,n)}\in \mathbb{R}$ for all
$n\in \mathbb{N}_{0}$. Then (\ref{smnsnmstar}) holds for all
$m,n\in \mathbb{N}$. To prove the $*$-invariance of $L_{S}$, as in the
proof of Lemma~\ref{twosidedls}, it suffices to take
$f=a (q^{*})^{m} q^{n} \sigma $. Then
$f^{*}=\sigma ^{*} (q^{*})^{n}q^{m} a^{*}$. Using the left and right
$\mathbb{H}$-linearity of $L_{S}$ and (\ref{smnsnmstar}) we derive
\begin{align*}
L(f^{*}) &=\sigma ^{*} L_{S}( (q^{*})^{n}q^{m}) a^{*}=\sigma ^{*}s_{(n,m)}
a^{*}
\\
&=\sigma ^{*}s_{(m,n)}^{*} a^{*}= (a s_{(m,n)}\sigma )^{*}=L_{S}(f)^{*},
\end{align*}
which proves that $L_{S}$ is $*$-invariant.
\end{proof}

From Lemma~\ref{twosidedls} and Lemma~\ref{lsstarinvariant} we obtain the
following proposition.
\begin{proposition}%
\label{Riesztwosided}
The Riesz functional $L_{S}$ of a quaternionic sequence
$S=\{s_{(m,n)}\}_{(m,n) \in \mathsf{{N}}}$ is a $*$-invariant two-sided
$\mathbb{H}$-linear functional on $\mathbb{H}^{0}[q^{*}, q]$ if and only
if the quaternion $s_{(n,n)}$ is real for all $n\in \mathbb{N}_{0}$.
\end{proposition}

\begin{definition}%
\label{momentsequence}
Let $K$ be a closed subset of $\mathbb{H}\cong \mathbb{R}^{4}$. A quaternionic
sequence $S=\{s_{(m,n)}\}_{(m,n) \in \mathsf{{N}}}$ is called a
\emph{$K$-moment sequence} if there exists a Radon measure $\mu $ on
$\mathbb{H}$ such that $\mathrm{supp}\mu \subseteq K$ and
\begin{equation}
\label{kmoments}
s_{(m,n)} = \int _{\mathbb{H}} (q^{*})^{m} q^{n}\, d\mu (q)\quad
\text{for all}~~~ (m,n) \in \mathsf{{N}.
}\end{equation}
If (\ref{kmoments}) holds, then the quaternion $s_{(m,n)} $ is called the
\emph{$(m,n)$-th moment }of the measure $\mu $ on $\mathbb{H}$ and
$\mu $ is called a \emph{representing measure} of $S$.
\end{definition}

Suppose $S$ is a $K$-moment sequence. Then, setting $m=n$, it follows from
(\ref{kmoments}) that $s_{(n,n)} \in \mathbb{R}$ for all
$n\in \mathbb{N}$. Hence, by Proposition~\ref{Riesztwosided}, the Riesz
functional $L_{S}$ is a $*$-invariant two-sided $\mathbb{H}$-linear functional
on $\mathbb{H}^{0}[q^{*}, q]$. This fact follows also easily from the next
lemma.

\begin{lemma}%
\label{Riesz}
A quaternionic sequence $S=\{s_{(m,n)}\}_{(m,n) \in \mathsf{{N}}}$ is a
$K$-moment sequence if and only if its Riesz functional $L_{S}$ is a
$K$-moment functional on $\mathbb{H}^{0}[q^{*}, q]$ according to Definition~\ref{quatmf}.
\end{lemma}
\begin{proof}
First suppose that $S$ is a $K$-moment sequence. Then, by Definition~\ref{momentsequence}, there is a Radon measure $\mu $ supported by
$K$ such that (\ref{kmoments}) holds. Let $f$ be as in (\ref{canonicalrep1}).
Then we derive
\begin{align*}
L_{S}(f) &= \sum _{(m,n,\sigma )\in \mathcal{N}} a_{(m,n,\sigma )}s_{(m,n)}
\sigma
\\
&= \sum _{(m,n,\sigma )\in \mathcal{N}} a_{(m,n,\sigma )} \left (
\int _{\mathbb{H}} (q^{*})^{m} q^{n}\, d\mu (q)\right ) \sigma
\\
&=\int _{\mathbb{H}} \left (\sum _{(m,n,\sigma )\in \mathcal{N}} a_{(m,n,
\sigma )} (q^{*})^{m} q^{n} \sigma \right )\, d\mu (q)=\int _{
\mathbb{H}} f(q) d\mu (q) .
\end{align*}
This shows that $L_{S}$ is indeed a $K$-moment functional.

Conversely, if $L_{S}$ is a $K$-moment functional, for
$(m,n) \in \mathsf{{N}}$ we have
\begin{align*}
s_{(m,n)}=L_{S}((q^{*})^{m} q^{n})= \int _{K} (q^{*})^{m} q^{n}\, d
\mu (q).
\end{align*}
That is, $S$ is a $K$-moment sequence.
\end{proof}

The following theorem is our main result about the quaternionic moment
problem for the basic polynomial space $\mathbb{H}^{0}[q^{*}, q]$.
\begin{theorem}%
\label{qmp1}
  Suppose that $K$ is a closed subset of
$\mathbb{H}\cong \mathbb{R}^{4}$ and let
$S=\{s_{(m,n)}\}_{(m,n) \in \mathsf{{N}}}$ be a quaternionic sequence. Then the
following are equivalent:%
\begin{enumerate}
\item[\emph{(i)}] $S$ is a $K$-moment sequence.
\item[\emph{(ii)}] $L_{S}$ is a $K$-moment functional on
$ \mathbb{H}^{0}[q^{*}, q]$.
\item[\emph{(iii)}]
\begin{equation}
L_{S}(f)\geq 0
\label{eq55}
\end{equation}
for all polynomials $f:=f_{0} +x_{1}f_{1}+x_{2}f_{2} +x_{3} f_{3}$, where
$f_{0}, f_{1},f_{2},f_{3}\in \mathbb{R}[x_{0}, x_{1}^{2}+x_{2}^{2}+x_{3}^{2}]$,
such that $f(x)\geq 0$ for all $x\in K$.
\end{enumerate}
\end{theorem}
\begin{proof}
The equivalence of (i) and (ii) is Lemma~\ref{Riesz}.

(ii)$\Leftrightarrow $(iii): By Proposition~\ref{adaptedH_L},
$ \mathbb{H}^{0}[q^{*}, q]_{\mathrm{her}}\lceil K$ is an adapted space
according to Definition~\ref{adapted}. Hence the quaternionic Choquet Theorem~\ref{quatchoquet} applies to the subspace
$E:= \mathbb{H}^{0}[q^{*}, q]\lceil K$ of $C(K;\mathbb{H})$. From the description
of the hermitian part given in Theorem~\ref{realpartH_L} we conclude that
(iii) is just the $E_{+}$-positivity condition of the functional
$L_{S}$. Therefore the equivalence of (ii) and (iii) follows from Theorem~\ref{quatchoquet}, (i)$\Leftrightarrow $(iii).
\end{proof}

To apply Theorem~\ref{qmp1} one has to verify condition (iii). This requires
a description of the set $E_{+}$ of basic quaternionic polynomials which
are nonnegative on the set $K$. In general, this might be a difficult task.
However, we will show that for specific sets $K$ one can determine the
set $E_{+}$ of positive elements and obtain solvability criteria for the
moment problem from Theorem~\ref{qmp1}.

First we prove an auxiliary lemma.

\begin{lemma}%
\label{auxlemma}
Let $a_{0},a_{1},\ldots ,a_{k},c\in \mathbb{R}$ such that not all
$a_{1},\ldots , a_{k}$ are zero. Set
\begin{align}
f(x) &:= a_{0}+a_{1}x_{1}+\ldots +a_{k}x_{k},
\label{fedfi}
\\
S &:=\{(x_{1},\ldots ,x_{k})\in \mathbb{R}^{k}: x_{1}^{2}+\ldots +x_{k}^{2}=c
\}.
\nonumber
\end{align}
Then the following conditions are equivalent:
\begin{enumerate}
\item[\emph{(i)}] $f(0)\geq 0$ and $f(x)\geq 0$ for all $x\in S$.
\item[\emph{(ii)}] $a_{0}\geq 0$ and
$a_{0}^{2} \geq c \, (a_{1}^{2}+\ldots +a_{k}^{2})$.
\end{enumerate}
\end{lemma}

\begin{proof}
First, we consider the case $c\geq 0$.

Assume that the conditions in (ii) are satisfied. For
$(x_{1},\ldots ,x_{k})\in S$, we derive
\begin{align}
f(x) &\geq a_{0}-\Big|\sum _{j=1}^{k} a_{j}x_{j}\Big|\geq a_{0}-\Big(
\sum _{j=1}^{k} a_{j}^{2}\Big)^{1/2}\Big(\sum _{j=1}^{k} x_{j}^{2}
\Big)^{1/2}
\nonumber
\\
&= a_{0}-\Big(\sum _{j=1}^{k} a_{j}^{2}\Big)^{1/2}\, c^{1/2}.
\label{fx0}
\end{align}
Therefore, since $f(0)=a_{0}\geq 0$ and
$a_{0}^{2} \geq c \, (a_{1}^{2}+\ldots +a_{k}^{2})$, it follows from \eqref{fx0} that $f(x)\geq 0$.

Conversely, assume that (i) holds. Set
$x_{j}=c^{1/2}\, a_{j}(a_{1}^{2}+\ldots +a_{k}^{2})^{-1/2}$. Then
$(x_{1},\ldots ,x_{k})\in S$ and
$f(x)=a_{0}-c^{1/2} (a_{1}^{2}+\ldots +a_{k}^{2})^{1/2}$. Then
$f(0)\geq 0$ implies that $a_{0} \geq 0$ and $f(x)\geq 0$ implies that
$a_{0}^{2} \geq c \, (a_{1}^{2}+\ldots +a_{k}^{2})$.

Let us consider the case $c<0$. Then $S$ is empty, so each linear functional
$f$ is obviously nonnegative on $S$ and the condition
$a_{0}^{2} \geq c \, (a_{1}^{2}+\ldots +a_{k}^{2})$ holds. Since
$f(0)=a_{0}$, (i) is trivially equivalent to (ii). Thus the assertion holds
also when $c<0$.
\end{proof}

In the remaining part of this section we consider sets of the form
\begin{align}
\label{defK}
K_{h}=\{x=(x_{0},x_{1},x_{2},x_{3})\in \mathbb{R}^{4}: x_{1}^{2}+x_{2}^{2}+x_{3}^{2}=h(x_{0})
\},
\end{align}
where $h\in \mathbb{R}[x_{0}]$ is a fixed polynomial. We study the quaternionic
moment problem for basic quaternionic polynomials on the set $K_{h}$.

Let $\widetilde{E}$ denote the vector space of restrictions of elements
of $\mathbb{H}^{0}[q^{*},q]_{\mathrm{her}}$ to the set $K _{h}$. On the
set $K_{h}$, the expression $x_{1}^{2}+x_{2}^{2}+x_{3}^{2}$ is equal to
$h(x_{0})$. Therefore, it follows from Theorem~\ref{realpartH_L} that
$\widetilde{E}$ coincides with the restriction to $K_{h}$ of
\begin{equation*}
E:=\mathbb{R}[x_{0}]+x_{1}\mathbb{R}[x_{0}]+x_{2}\mathbb{R}[x_{0}]+x_{3}
\mathbb{R}[x_{0}].
\end{equation*}

The next lemma describes the set
\begin{equation*}
E_{+}:=\{ g\in E: g(x)\geq 0 \text{ for all } x\in K_{h}\}
\end{equation*}
of nonnegative polynomials on $K_{h}$. Define
\begin{equation*}
M_{h}:=\{ x_{0}: (x_{0},x_{1},x_{2},x_{3}) \in K_{h}
\text{ for some } x_{1},x_{2},x_{3}\in \mathbb{R}\}=h^{-1}(\mathbb{R}^{+}).
\end{equation*}

\begin{lemma}%
\label{positiveE}
Let
$g=p_{0}(x_{0})+x_{1}p_{1}(x_{0})+x_{2}p_{2}(x_{0})+x_{3}p_{3}(x_{0})
\in E$, where $p_{0},p_{1},p_{2},p_{3}\in \mathbb{R}[x_{0}]$. Then
$g\in E_{+}$ if and only if
\begin{align}
& p_{0}(x_{0}) \geq 0\quad \text{for}\quad x_{0}\in M_{h} \quad
\text{and}
\label{pestimate1}
\\
& p_{0}(x_{0})^{2}\geq h(x_{0})(p_{1}(x_{0})^{2}+p_{2}(x_{0})^{2}+p_{3}(x_{0})^{2})
\quad \text{for}\quad x_{0}\in M_{h}.
\label{pestimate2}
\end{align}
\end{lemma}

\begin{proof}
First suppose that $g\in E_{+}$. Let
$x=(x_{0},x_{1},x_{2},x_{3})\in K_{h}$. Then
$(x_{0},-x_{1},-x_{2},\allowbreak -x_{3})\in K_{h}$ and
\begin{align}
\label{fprel}
g(x_{0},x_{1},x_{2},x_{3})+g(x_{0},-x_{1},-x_{2},-x_{3})=2p_{0}(x_{0}).
\end{align}
Hence $p_{0}(x_{0})\geq 0$. Define $f$ by \eqref{fedfi}, with
$a_{0}=p_{0}(x_{0})$, $a_{1}=p_{1}(x_{0})$, $a_{2}=p_{2}(x_{0})$,
$a_{3}=p_{3}(x_{0})$, $c=h(x_{0})$ and $k=3$. It follows from Lemma~\ref{auxlemma}, (i)$\Rightarrow $(ii), that \eqref{pestimate2} holds.

Conversely, assume \eqref{pestimate1} and \eqref{pestimate2}. Then
$p_{0}(x_{0})=a_{0}\geq 0$ by \eqref{fprel} and \eqref{pestimate1} and
therefore $g\in E_{+}$ by Lemma~\ref{auxlemma}, (ii)$\Rightarrow $(i).
\end{proof}

Combining Theorem~\ref{qmp1} and Lemma~\ref{positiveE} gives the following
solvability criterion for the $K_{h}$-moment problem.

\begin{theorem}%
\label{mpK_h}
Suppose that $L$ is an $\mathbb{H}$-linear functional on
$\mathbb{H}^{0}[q^{*},q]$. Let $h\in \mathbb{R}[x_{0}]$ and let
$K_{h}$ be the real algebraic set defined by \eqref{defK}. Then $L$ is
a $K_{h}$-moment functional if and only if
\begin{align}
L(g)\geq 0
\label{eq62}
\end{align}
for all elements $g:=p_{0}+x_{1}p_{1}+x_{2}p_{2}+x_{3}p_{3}\in E$, where
$p_{0},p_{1},p_{2},p_{3}\in \mathbb{R}[x_{0}]$ satisfy \eqref{pestimate1} and \eqref{pestimate2}.
\end{theorem}

Let us mention a few interesting special cases of Theorem~\ref{mpK_h}:
\begin{enumerate}
\item[1.)] $h(x_{0})=1$.%

Then $K_{h}$ is a cylinder in $\mathbb{R}^{4}$ with the unit sphere of
$\mathbb{R}^{3}$ as base set.%

In this case, we have $M_{h}=\mathbb{R}$ and Lemma~\ref{positiveE} reads
\begin{equation*}
E_{+}= \{ p_{0}+x_{1}p_{1}+x_{2}p_{2}+x_{3}p_{3}\in E \ : \: p_{0}(x_{0})^{2}
\ge p_{1}(x_{0})^{2}+p_{2}(x_{0})^{2}+p_{3}(x_{0})^{2}\}.
\end{equation*}
In the other cases, even if it is easy to find $M_{h}$, but Lemma~\ref{positiveE} does not give such simple description of $E_{+}$.
\item[2.)] $h(x_{0})=1-x_{0}^{2}$.

Then $K_{h}$ is the unit sphere of $\mathbb{R}^{4}$.
\item[3.)] $h(x_{0})=1+x_{0}^{2}$.%

Then $K_{h}$ is a hyperboloid in $\mathbb{R}^{4}$.
\item[4.)] $h(x_{0})=x_{0}-x_{0}^{2}$.
\item[5.)] $h(x_{0})=x_{0}+x_{0}^{2}$.%

In cases 4.) and 5.), $K_{h}$ is a parabolic set in $\mathbb{R}^{4}$.
\end{enumerate}

\section{The quaternionic moment problem for polynomial spaces $ \mathbb{H}^{k}[q^{*}, q]$}
\label{mphk}

To formulate a version of Haviland's theorem for quaternions it is convenient
to introduce some more notions.

Let $K$ be a closed subset of $\mathbb{H}\cong \mathbb{R}^{4}$. We define
\begin{align*}
E^{k}(K)_{+}:=\Big\{f\in \sum _{\alpha \in \mathbb{N}_{0}^{3}, |
\alpha |\leq k+1} x^{\alpha}\, \mathbb{R}[x_{0},x_{1}^{2}+x_{2}^{2}+x_{3}^{3}]:
f(x)\geq 0~~ \textrm{for}~~ x\in K \Big\}
\end{align*}
for $k\in \mathbb{N}_{0}$ and
\begin{align*}
E^{\infty}(K)_{+}:= \big\{ f\in \mathbb{R}[x_{0},x_{1},x_{2},x_{3}]: f(x)
\geq 0~~ \textrm{for}~~ x\in K\, \big\}.
\end{align*}

\begin{definition}%
\label{kpositivity}
Let $k\in \mathbb{N}_{0}\cup \{\infty \}$. An $\mathbb{H}$-linear functional
$L: H^{k}[q^{*}, q] \to \mathbb{H}$ is called
\begin{itemize}
\item[$\bullet $] \emph{$K$-positive} if $L(f)\geq 0$ for all
$f\in E^{k}(K)_{+}$.

\item[$\bullet $] \emph{weakly $K$-positive} if for each
$f\in E^{k}(K)_{+}$ there exists an $h\in E^{k}(K)_{+}$ such that
$L(f+\varepsilon h) \geq 0$  for all $\varepsilon >0$.
\end{itemize}
\end{definition}

Clearly, if $L$ is $K$-positive, then $L$ is also weakly $K$-positive.
Note that the weak $K$-positivity is not a standard notion. We use it in
order to have a convenient formulation of the following theorem.

The following result can be considered as a quaternionic version of Haviland's
theorem, see  \cite{Haviland1936}.
\begin{theorem}
\label{thm:quaternion_haviland}
Suppose that $k\in \mathbb{N}_{0}\cup \{\infty \}$ and $K$ is a closed
subset of $\mathbb{H}\cong \mathbb{R}^{4}$. Let
$L: \mathbb{H}^{k}[q^{*}, q] \to \mathbb{H}$ be an $\mathbb{H}$-linear
functional. Then the following statements are equivalent:
\begin{enumerate}
\item[\emph{(i)}] The functional $L$ is $K$-positive.
\item[\emph{(ii)}] The functional $L$ is weakly $K$-positive.
\item[\emph{(iii)}] $L$ is a $K$-moment functional on
$\mathbb{H}^{k}[q^{*}, q]$, that is, there exists a Radon measure
$\mu $ on $\mathbb{H}$ such that $\mathrm{supp}\, \mu \subseteq K$ and
\begin{equation*}
L(f) = \int _{\mathbb{H}}\, f(q) d\mu (q)\quad \text{for all} ~~~ f
\in \mathbb{H}^{k}[q^{*}, q].
\end{equation*}
\end{enumerate}%
\end{theorem}

\begin{proof}
By Proposition~\ref{adaptedH_L}, the space
${\mathbb{H}^{k}[q^{*},q]_{\mathrm{her}}\lceil K}$ forms an adapted linear
subspace of $C(K; \mathbb{R})$ according to Definition~\ref{adapted}. Therefore,
the quaternionic Choquet Theorem~\ref{quatchoquet} applies to
$E:=\mathbb{H}^{k}[q^{*}, q] \lceil K$. From Theorems~\ref{herhk} and \ref{herinfty} it follows that the set $E^{k}(K)_{+}$ defined above is
just the set of elements of
${\mathbb{H}^{k}[q^{*},q]_{\mathrm{her}}\lceil K}$ that are non-negative
on $K$. Therefore, because of Definition~\ref{kpositivity}, Theorem~\ref{quatchoquet} gives the equivalence of conditions (i)--(iii).
\end{proof}%
\begin{remark}
\label{rem6.3}
For $k\in \mathbb{N}_{0}$, Theorem~\ref{herhk} implies that
$x_{1}\in \mathbb{H}^{k}[q^{*}, q]_{\mathrm{her}}$ and
$x_{1}^{k+2}\notin \mathbb{H}^{k}[q^{*}, q]_{\mathrm{her}}$. Hence the
real vector space $ \mathbb{H}^{k}[q^{*}, q]_{\mathrm{her}}$ is \textit{not}
an algebra and it is difficult to describe the hermitian elements that
are positive on $K$. In order to apply methods of real algebraic geometry
a finitely generated real algebra is needed. However, note that the hermitian
part $ \mathbb{H}^{k}[q^{*}, q]_{\mathrm{her}}$ is a module for the algebra
$\mathbb{R}[x_{0},x_{1}^{2}+x_{2}^{2}+x_{3}^{2}]$.
\end{remark}

\section{The quaternionic moment problem for $ \mathbb{H}^{\infty}[q^{*}, q]$}
\label{mpinfty}

In this section we consider the moment problem for the full polynomial
algebra $ \mathbb{H}^{\infty}[q^{*}, q]$. In this case, by Theorem~\ref{herinfty}, the hermitian part
$\mathbb{H}^{\infty}[q^{*},q]_{\mathrm{her}}$ is the real polynomial algebra
$\mathbb{R}[x_{0},x_{1}, x_{2},x_{3}]$. Therefore, using Theorem~\ref{thm:quaternion_haviland} in the case $k=\infty $, methods and results
from real algebra (see e.g., \cite{Marshall2008},
\cite{Scheiderer2024}, \cite[Chapter 12]{Schmudgen2017}) can be applied.

Roughly speaking, each version of the Archimedean Positivstellensatz for
the polynomial algebra $\mathbb{R}[x_{0},x_{1},x_{2},x_{3}]$ provides a
solvability criterion for the moment problem. Let $C$ be a cone of
$\mathbb{R}[x_{0},x_{1},x_{2},x_{3}]$ and $K$ a subset of
$\mathbb{H}\cong \mathbb{R}^{4}$ for which the assertion of the Archimedean
Positivstellensatz holds. This means that each polynomial
$f \in \mathbb{R}[x_{0},x_{1},x_{2},x_{3}]$, such that $f(x)>0$ for all
$x\in K$, belongs to $C$. Suppose that
$L:\mathbb{H}^{\infty}[q^{*}, q]\to \mathbb{H}$ is an $\mathbb{H}$-linear
functional such that $L(f)\geq 0$ for all $f\in C$. Then $L$ is weakly
$K$-positive according to Definition~\ref{kpositivity}. Indeed, setting
$h=1$ in Definition~\ref{kpositivity}, we have
$f(q)+\varepsilon h(q)>0$ for all $f\in E^{\infty}(K)_{+}$,
$\varepsilon >0$ and $q\in K$. Then, $f+\varepsilon h\in C$ by the Archimedean
Positivstellensatz and hence $L(f+\varepsilon h)\geq 0$.

There are several Archimedean Positivstellens\"atze for quadratic modules,
semirings and preorderings (see \cite{Marshall2008},
\cite{Schmudgen2017}, \cite{Scheiderer2024}) and each of these results
yields a solvability criterion for the classical moment problem. In fact,
all such results about the moment problem for
$\mathbb{R}[x_{0},x_{1},x_{2},x_{3}]$ developed in
\cite[Chapter 12]{Schmudgen2017} can be reformulated for the quaternionic
moment problem on the polynomial space
$\mathbb{H}^{\infty}[q^{*}, q]$. We do not carry out this and restrict
ourselves to the pioneering theorem for compact semi-algebraic sets proved
in \cite{Schmudgen1991}.

Assume that $\mathsf{{f}}$ is a $k$-tuple of polynomials
$f_{1},\dots ,f_{k}\in \mathbb{R}[x_{0},x_{1},x_{2},x_{3}]$. Then
\begin{align}
\label{defsemialgebraicset}
K({\mathsf{f}}):=\{x\in \mathbb{R}^{4}: f_{1}(x)\geq 0,\dots ,f_{k}(x)
\geq 0\}
\end{align}
is called the \emph{basic closed semi-algebraic set} associated with
$\mathsf{f}$. Further, the set $T({\mathsf{f}})$ of all finite sums of terms
$ f_{1}^{e_{1}}\cdots f_{k}^{e_{k}}g^{2}$, where
$ e_{1},\dots ,e_{k}\in \{0,1\}$ and
$g\in \mathbb{R}[x_{0},x_{1},x_{2},x_{3}]$, is called the
\emph{preordering} generated by $\mathsf{f}$.
\begin{theorem}%
\label{mpppreordering}
Suppose that the set $K({\mathsf{f}})$ is compact. Let
$L:\mathbb{H}^{\infty}[q^{*}, q]\to \mathbb{H}$ be an $\mathbb{H}$-linear
functional. Then $L$ is a $K({\mathsf{f}})$-moment functional if and only if
\begin{align}
\label{preorder}
L( f_{1}^{e_{1}}\cdots f_{k}^{e_{k}}g^{2})\geq 0
\end{align}
for all $ e_{1},\dots ,e_{k}\in \{0,1\}$ and
$g\in \mathbb{R}[x_{0},x_{1},x_{2},x_{3}]$.
\end{theorem}
\begin{proof}
Condition (\ref{preorder}) means that $L$ is nonnegative on the preordering
$T({\mathsf{f}})$ Hence the assertion follows from Theorem~\ref{thm:quaternion_haviland} combined with
\cite[Theorem 1]{Schmudgen1991} or
\cite[Theorem 12.25]{Schmudgen2017}.
\end{proof}

Let us state condition (\ref{preorder}) in the case $k=2$. Then (\ref{preorder})
means that
\begin{align*}
L(g^{2})\geq 0,~~ L(f_{1}g^{2})\geq 0, ~ ~ L(f_{2}g^{2})\geq 0,~~~~
\text{and}~~ L(f_{1}f_{2}g^{2})\geq 0
\end{align*}
for all $g\in \mathbb{R}[x_{0},x_{1},x_{2},x_{3}]$. We illustrate this
by a simple example.
\begin{example}
\label{exmp7.2}
Let $k=2$ and $f_{1}=1-q^{*}q$, $f_{2}=q^{*} \mathrm{i}-\mathrm{i}q$. Then
an $\mathbb{H}$-linear functional
$L:H^{\infty}[q^{*}, q]\to \mathbb{H}$ is a $K({\mathsf{f}})$-moment functional
if and only if
\begin{align}
\label{threeconditions}
L(g^{2})\geq 0,& ~~ L((1-q^{*}q)g^{2})\geq 0, ~ ~ L((q^{*} \mathrm{i}-
\mathrm{i}q) g^{2})\geq 0,
\\
& L((1-q^{*}q)(q^{*} \mathrm{i}-\mathrm{i}q)g^{2})\geq 0
\label{fourthcondition}
\end{align}
for all $g\in \mathbb{R}[x_{0},x_{1},x_{2},x_{3}]$. Note that in this particular
case, it turns out that the condition (\ref{fourthcondition}) can be omitted,
that is, $L$ is a $K({\mathsf{f}})$-moment functional if and only if the three
conditions in (\ref{threeconditions}) are satisfied.
\end{example}

The solvability condition (\ref{preorder}) can be nicely reformulated in
terms of \textit{localized Hankel matrices}. For this we recall the standard
multi-index notation
\begin{align*}
x^{\alpha }= x_{0}^{\alpha _{0}} x_{1}^{\alpha _{1}}\, x_{2}^{\alpha _{2}}
\, x_{3}^{\alpha _{3}}\quad \textrm{for}\, ~~~ \alpha =( \alpha _{0},
\alpha _{1},\alpha _{2},\alpha _{3})\in \mathbb{N}_{0}^{4}, ~ x_{j}^{0}:=1.
\end{align*}

Suppose that $L: \mathbb{H}^{\infty}[q^{*}, q] \to \mathbb{H}$ is an
$\mathbb{H}$-linear functional. Let
$f=\sum _{\alpha }f_{\alpha }x^{\alpha}$ be a polynomial of
$ \mathbb{R}[x_{0},x_{1},x_{2},x_{3}]$. For
$\alpha \in \mathbb{N}_{0}^{4}$, we set
$s_{\alpha }= L(x^{\alpha}) $. The matrix $H(fL)$ with entries
\begin{equation}
H(fL)_{\alpha ,\beta}:=\sum \nolimits _{\gamma }\, f_{\gamma }s_{
\alpha + \beta +\gamma},\quad \alpha , \beta \in \mathbb{N}_{0}^{4},
\label{eq67}
\end{equation}
is called the localized Hankel matrix of $L$ at $f$. For
$ g=\sum \nolimits _{\alpha }g_{\alpha }x^{\alpha}\in \mathbb{R}[x_{0},x_{1},x_{2},x_{3}]$
we compute
\begin{align}
\label{hgl}
L(fg^{2})=\sum \nolimits _{\alpha , \beta} H(fL)_{\alpha ,\beta}\, g_{
\alpha }g_{\beta}.
\end{align}
Therefore, comparing (\ref{preorder}) and (\ref{hgl}) we conclude (see
also \cite[Proposition 12.14]{Schmudgen2017}) that condition (\ref{preorder})
holds if and only if the localized Hankel matrices
$H(f_{1}^{e_{1}}\cdots f_{k}^{e_{k}}L)$ are positive semi-definite for
$e_{1},\dots ,e_{k}\in \{0,1\}$. Thus, Theorem~\ref{mpppreordering} yields
the following corollary.
\begin{corollary}
\label{cor7.3}
Suppose that the semi-algebraic set $K({\mathsf{f}})$ defined by (\ref{defsemialgebraicset})
is compact. Then an $\mathbb{H}$-linear functional
$L:\mathbb{H}^{\infty}[q^{*}, q]\to \mathbb{H}$ is a $K({\mathsf{f}})$-moment
functional if and only if the localized Hankel matrices
$H(f_{1}^{e_{1}}\cdots f_{k}^{e_{k}}L)$ are positive semi-definite for
all $ e_{1},\dots ,e_{k}\in \{0,1\}$.
\end{corollary}

\section*{Acknowledgments}
The first author is partially supported by the Arab Fund Foundation Fellowship Program under the Distinguished Scholar Award \# 1074. We would like to thank Leipzig University for hosting this research.

\subsection*{Data availability}
No data was used for the research described in the article.


\begin{thebibliography}{99}

\bibitem{Akhiezer1965} N.I. Akhiezer, \emph{The Classical Moment Problem and Some Related Questions in Analysis}, Hafner Publishing Co., New York, 1965.

\bibitem{AP21} G. Alon and E. Parad, A quaternionic nullstellensatz, J. Pure Appl. Algebra \textbf{225}(2021), Nr. 106572.

\bibitem{Choquet1969} G. Choquet, \emph{Lectures on Analysis}, W.A. Benjamin, New York, 1969.

\bibitem{CGK} F. Colombo, G, J. Gantner and D. Kimsey, \emph{Spectral Theory on the $S$-Spectrum for Quaternionic Operators}, Operator Theory: Advances and Applications \textbf{270}, Birkh\"auser, 2018.

\bibitem{GSS} G. Gentili, C. Stoppato and D.C. Struppa, \emph{Regular Functions of a Quaternion Variable}, Springer Monographs in Mathematics, 2013.

\bibitem{Hamilton1843} W.R. Hamilton, On quaternions, or on a new system of imaginaries in algebra, The London, Edinburgh and Dublin Philosophical Magazine and Journal of Science \textbf{25} (1843), 489--495.

\bibitem{Haviland1936} E.K. Haviland, On the moment problem for distribution functions in more than one dimension II, Amer. J. Math. \textbf{58} (1936), 164--168.

\bibitem{Marshall2008} M. Marshall, \emph{Positive Polynomials and Sums of Squares}, Amer. Math. Soc., Providence, R.I., 2008.

\bibitem{rodman2014} L. Rodman, \emph{Topics in Quaternion Linear Algebra}, Princeton University Press, 2014.

\bibitem{Scheiderer2024} C. Scheiderer, \emph{A Course in Real Algebraic Geometry}, Graduate Texts in Mathematics \textbf{303}, Springer, Cham, 2024.

\bibitem{Schmudgen1991} K. Schm\"{u}dgen, The $K$-moment problem for compact semi-algebraic sets, Math. Ann. \textbf{289} (1991), 203--206.

\bibitem{Schmudgen2017} K. Schm\"{u}dgen, \emph{The Moment Problem}, Graduate Texts in Mathematics \textbf{277}, Springer, Cham, 2017.

\bibitem{Schmudgen2020} K. Schm\"{u}dgen, \emph{An Invitation to Unbounded $*$-Representations of $*$-Algebras on Hilbert Space}, Graduate Texts in Mathematics \textbf{285}, Springer, Cham, 2020.

\bibitem{Stieltjes1894} T.J. Stieltjes, Recherches sur les fractions continues, Ann. Fac. Sci. Toulouse Math. \textbf{8} (1894), 1--122.

\bibitem{vasi22} Vasilescu, F., Functions and operators in real, quaternionic and Cliffordian context, Complex Anal. Oper. Theory \textbf{16}(2022), No. 8, 117.
\end{thebibliography}
\end{document}